\RequirePackage{ifpdf}
\ifpdf % We are running pdfTeX in pdf mode
\documentclass[pdftex]{sigma}
\else
\documentclass{sigma}
\fi

\usepackage{epic} % Better commands for pictures; in particular \drawline
\usepackage{lscape} % Allows single pages to be set in landscape
\usepackage{calc} % Adds support for calculating with counters (allowing me to write \connect).

\newcommand{\rphis}[2]{{}_{#1\vphantom{#2}}\phi_{#2\vphantom{#1}}}
\newcommand{\rphisn}[3]{{}_{#1\vphantom{#2}}^{\vphantom{(#3)}}\phi_{#2\vphantom{#1}}^{(#3)}}
\newcommand{\rWs}[2]{{}_{#1\vphantom{#2}}W_{#2\vphantom{#1}}}
\newcommand{\rWsn}[3]{{}_{#1\vphantom{#2}}^{\vphantom{(#3)}}W_{#2\vphantom{#1}}^{(#3)}}

\newcommand{\rphisx}[4]{\rphis{#1}{#2}\left( \begin{array}{c} #3 \end{array};q,#4\right)}
\newcommand{\rphisxk}[5]{\rphisn{#1}{#2}{#3}\left( \begin{array}{c} #4 \end{array};q,#5\right)}
\newcommand{\rphisxkn}[3]{\rphisn{}{}{#1}\left( \begin{array}{c} #2 \end{array};q,#3\right)}

\newcommand{\ditto}{\prime\prime}

\newcounter{xtemp1}
\newcounter{ytemp1}
\newcounter{xtemp2}
\newcounter{ytemp2}
\newcommand{\mput}[2]{\put(#1){\makebox(0,0){#2}}}
\newcommand{\connect}[5]{
 \setcounter{xtemp1}{#1+ (#2-#1) * #4/30}
 \setcounter{xtemp2}{#2+(#1-#2)*#5/30}
 \setcounter{ytemp1}{#3+#4}
 \setcounter{ytemp2}{#3+30-#5}
 \drawline(\arabic{xtemp1},\arabic{ytemp1})(\arabic{xtemp2}, \arabic{ytemp2}) %Brilliantly use \arabic{counter} to turn it into a number read as argument. I don't know why \value{counter} won't work.
}

\DeclareMathOperator{\Stab}{Stab}
\DeclareMathOperator{\Res}{Res}

\numberwithin{equation}{section}

\newtheorem{Theorem}{Theorem}[section]
\newtheorem{Lemma}[Theorem]{Lemma}
\newtheorem{Proposition}[Theorem]{Proposition}
{\theoremstyle{definition}
\newtheorem{Definition}[Theorem]{Definition}
\newtheorem{Remark}[Theorem]{Remark}
}

\begin{document}

%\allowdisplaybreaks

\renewcommand{\thefootnote}{$\star$}

\renewcommand{\PaperNumber}{059}

\FirstPageHeading

\ShortArticleName{Basic Hypergeometric Functions as Limits of Elliptic Hypergeometric Functions}

\ArticleName{Basic Hypergeometric Functions\\ as Limits of Elliptic Hypergeometric Functions\footnote{This paper is a contribution to the Proceedings of the Workshop ``Elliptic Integrable Systems, Isomonodromy Problems, and Hypergeometric Functions'' (July 21--25, 2008, MPIM, Bonn, Germany). The full collection
is available at
\href{http://www.emis.de/journals/SIGMA/Elliptic-Integrable-Systems.html}{http://www.emis.de/journals/SIGMA/Elliptic-Integrable-Systems.html}}}

\Author{Fokko J. VAN DE BULT and Eric M. RAINS}

\AuthorNameForHeading{F.J. van de Bult and E.M. Rains}

\Address{MC 253-37, California Institute of Technology, 91125, Pasadena, CA, USA }
\Email{\href{mailto:vdbult@caltech.edu}{vdbult@caltech.edu}, \href{mailto:rains@caltech.edu}{rains@caltech.edu}}

\ArticleDates{Received February 01, 2009; Published online June 10, 2009}

\Abstract{We describe a uniform way of obtaining basic hypergeometric functions as limits of the elliptic beta integral.
This description gives rise to the construction of a polytope with a
dif\/ferent basic hypergeometric function attached to each face of this
polytope. We can subsequently obtain various relations, such as
transformations and three-term relations, of these functions by considering
geometrical properties of this polytope.
The most ge\-ne\-ral functions we describe in this way are sums of two
very-well-poised ${}_{10}\phi_9$'s and their Nassrallah--Rahman type
integral representation.}

\Keywords{elliptic hypergeometric functions, basic hypergeometric functions, transformation formulas}

\Classification{33D15}

\section{Introduction}

Hypergeometric functions have played an important role in mathematics, and
have been much studied since the time of Euler and Gau\ss. One of the goals
of this research has been to obtain hypergeometric identities, such as
evaluation and transformation formulas. Such formulas are of interest due
to representation-theoretical interpretations, as well as their use in
simplifying sums appearing in combinatorics.

In more recent times people have been trying to understand the structure
behind these formulas. In particular people have studied the symmetry
groups associated to certain hypergeometric functions, or the three terms
relations satisf\/ied by them (see \cite{LvdJ1} and \cite{LvdJ}).

Another recent development is the advent of elliptic hypergeometric
functions. This def\/ines a whole new class of hypergeometric functions, in
addition to the ordinary hypergeometric functions and the basic
hypergeometric functions. A nice recent overview of this theory is given in
\cite{Spiress}. For several of the most important kinds of formulas for
classical hypergeometric functions there exist elliptic hypergeometric
analogues. It is well known that one obtains basic hypergeometric functions
upon taking a limit in these elliptic hypergeometric functions. However a
systematic description of all possible limits had not yet been undertaken.

In this article we provide such a description of limits, extending work by
Stokman and the authors \cite{vdBRS}. This description provides some extra
insight into elliptic hypergeometric functions, as it indicates what
relations for elliptic hypergeometric functions correspond to what kinds of
relations for basic hypergeometric functions. Conversely we can now more
easily tell for what kind of relations there have not yet been found proper
elliptic hypergeometric analogues.

More importantly though, this description provides more insight into the
structure of basic hypergeometric functions and their relations, in the
form of a geometrical description of a large class of functions and
relations. All the results for basic hypergeometric functions we obtain
can be shown to be limits of previously known relations satisf\/ied by sums
of two very-well-poised ${}_{10}\phi_9$'s and their Nassrallah--Rahman like
integral representation. However, we would have been unable to place them
in a geometrical picture as we do in this article without considering these
functions as limits of an elliptic hypergeometric function.

In this article we focus on the (higher-order) elliptic beta integral
\cite{evalform}. For any $m\in \mathbb{Z}_{\geq 0}$ the function $E^m(t)$
is def\/ined for $t\in \mathbb{C}^{2m+6}$ satisfying the balancing condition
\[
\prod_{r=0}^{2m+5} t_r = (pq)^{m+1}
\]
by the formula
\[
E^m(t) = \Biggl(\prod_{0\leq r<s\leq 2m+5} (t_rt_s;p,q)\Biggr) \frac{(p;p)(q;q)}{2} \int_{\mathcal{C}} \frac{\prod\limits_{r=0}^{2m+5} \Gamma(t_r z^{\pm 1})}{\Gamma(z^{\pm 2})} \frac{dz}{2\pi i z}.
\]
Here $\Gamma$ denotes the elliptic gamma function and is def\/ined in Section
\ref{secnot}, as are the $(p,q)$-shifted factorials $(x;p,q)$.

Two important results for the elliptic beta integral are the existence of
an evaluation formula for $E^0$
and the fact that $E^1$ is invariant under an action of the Weyl group
$W(E_7)$ of type~$E_7$~\cite{e7trafo}. A more thorough discussion of the
elliptic beta integral is provided in Section~\ref{secellbeta}.

The main result of this paper is the following (see Theorems~\ref{thmlimitsexist}--\ref{thmorthface}), and its analogues for $m=0$, $m>1$.
\begin{Theorem}\label{th1}
Let $P$ denote the convex polytope in $\mathbb{R}^8$ with vertices
\[
e_i+e_j, \quad 0\leq i<j\leq 7, \qquad \frac12 \left(\sum_{r=0}^7 e_r\right) - e_i-e_j, \quad 0\leq i<j\leq 7.
\]
Then for each $\alpha \in P$ the limit
\[
B^1_{\alpha}(u) = \lim_{p\to 0} E^1\big(p^{\alpha_0} u_0, \ldots, p^{\alpha_7} u_7\big)
\]
exists as a function of $u\in \mathbb{C}^8$ satisfying the balancing
condition $\prod u_r = q^{2}$. Moreover, $B^1_{\alpha}$ depends only on
the face of the polytope which contains $\alpha$ and is a
function of the projection of $\log(u)$ to the space orthogonal to that face.
\end{Theorem}

\begin{Remark} The polytope $P$ was studied in an unrelated context in \cite{CS},
where it was referred to as the ``Hesse polytope'', as antipodal pairs of
vertices are in natural bijection with the bitangents of a plane quartic
curve.
\end{Remark}

As stated the theorem is rather abstract, but for each point in this
polytope we have an explicit expression of the limit as either a basic
hypergeometric integral, or a basic hypergeometric series, or a product of
$q$-shifted factorials (and sometimes several of these options). A graph
containing all these functions is presented in Appendix \ref{secapp1}. We
also obtain geometrical descriptions of various relations between these
limits $B^1_{\alpha}$.

Note that the vertices of the polytope are given by the roots satisfying $\rho\cdot u=1$
of the root system $R(E_8) = \{ u\in \mathbb{Z}^8 \cup (\mathbb{Z}^8 + \rho)
~|~ u\cdot u=2\}$, where $\rho=\{1/2\}^8$. In particular, the Weyl group
$W(E_7)=\Stab_{W(E_8)}(\rho)$ acts on the polytope in a natural way, which
is consistent with the $W(E_7)$-symmetry of $E^1$. As an immediate
corollary of this $W(E_7)$ invariance we obtain both the symmetries of
the limit $B^1_\alpha$ (determined by the stabilizer in $W(E_7)$ of the
face containing $\alpha$) and transformations relating dif\/ferent limits
(determined by the orbits of the face $\alpha$). Special cases of these
include many formulas found in Appendix III of Gasper and Rahman
\cite{GandR}. For example, they include Bailey's four term transformation
of very-well-poised ${}_{10}\phi_9$'s (as a symmetry of the sum of two
${}_{10}\phi_9$'s), the Nassrallah--Rahman integral representation of a
very-well-poised ${}_8\phi_7$ (as a transformation between two dif\/ferent
limits) and the expression of a very-well-poised ${}_8\phi_7$ in terms of
the sum of two ${}_4\phi_3$'s.

\looseness=1
Three term relations involving the dif\/ferent basic hypergeometric functions
can be obtained as limits of $p$-contiguous relations satisf\/ied by $E^1$
(and geometrically correspond to triples of points in $P$ dif\/fering by
roots of $E_7$), while the $q$-contiguous relations satisf\/ied by $E^1$ reduce to the
($q$-)contiguous relations satisf\/ied by its basic hypergeometric limits. In
particular, we see that these two qualitatively dif\/ferent kinds of formulas
for basic hypergeometric functions are closely related: indeed, they are
dif\/ferent limits of essentially the same elliptic identity!

A similar statement can be made for $E^0$, which leads to evaluation
formulas of its basic hypergeometric limits. Special cases of these
include Bailey's sum for a very-well-poised ${}_8\phi_7$ and the
Askey--Wilson integral evaluation.

We would like to remark that a similar analysis can be performed for
multivariate integrals. In particular the polytopes we obtain here are the
same as the polytopes we get for the multivariate elliptic Selberg
integrals (previously called type $I\!I$ integrals) of
\cite{vDS, vDS2, e7trafo,limits}.
In a future article the authors will also consider the limits of the
(bi-)orthogonal functions of~\cite{e7trafo}, generalizing and
systematizing the $q$-Askey scheme.

\looseness=1
The article is organized as follows. We begin with a small section on
notations, followed by a review of some of the properties of the elliptic
beta integrals. In Section~\ref{seclimits} we will describe the explicit
limits we consider. In Section~\ref{secpolytope} we def\/ine convex
polytopes, each point of which corresponds to a direction in which we can
take a limit. Moreover in this section we prove the main theorems of this
article, describing some basic properties of these basic hypergeometric
limits in terms of geometrical properties of the polytope. In Section~\ref{sectrafo} we harvest by considering the consequences in the case we
know non-trivial transformations of the elliptic beta integral. Section~\ref{Heine} is then devoted to explicitly giving some of these
consequences in an example, on the level of $\rphis{2}{1}$. Section~\ref{seceval} describes some
peculiarities specif\/ic to the evaluation $(E^0)$ case. Finally in Section~\ref{secrq} we consider some remaining questions, in particular focusing on
what happens for limits outside our polytope. The appendices give a graphical
representation of the dif\/ferent limits we obtain and a quick way of
determining what kinds of relations these functions satisfy.

\section{Notation}\label{secnot}

Throughout the article $p$ and $q$ will be complex numbers satisfying
$|p|,|q|<1$, in order to ensure convergence of relevant series and
products. Note that $q$ is generally assumed to be f\/ixed, while~$p$ may
vary.

We use the following notations for $q$-shifted factorials and theta functions:
\[
(x;q) = (x;q)_\infty = \prod_{j=0}^\infty (1-xq^j), \qquad
(x;q)_k = \frac{(x;q)_\infty}{(xq^k;q)_\infty},
\qquad \theta(x;q) = (x,q/x;q),
\]
where in the last equation we used the convention that $(a_1,\ldots,
a_n;q)= \prod_{i=1}^n (a_i;q)$, which we will also apply to gamma
functions. Moreover we will use the shorthand $(xz^{\pm 1};q) =
(xz,xz^{-1};q)$.

Many of the series we obtain as limits are conf\/luent, and in some cases,
highly conf\/luent. To simplify the description of such limits, we will use
a slightly modif\/ied version of the notation for basic hypergeometric series
in \cite{GandR}. In particular we set
\begin{gather*}
\rphisxk{r}{s}{n}{a_1,a_2,\ldots,a_r \\ b_1,b_2,\ldots,b_s}{z} =
\sum_{k=0}^\infty \frac{(a_1,a_2,\ldots,a_r;q)_k}{(q,b_1,b_2,\ldots,b_s;q)_k} z^k \left( (-1)^k q^{\binom{k}{2}} \right)^{n+s+1-r}.
\end{gather*}
In terms of the original ${}_r\phi_s$ from \cite{GandR} this is
\begin{gather*}
\rphisxk{r}{s}{n}{a_1,a_2,\ldots,a_r \\ b_1,b_2,\ldots,b_s}{z} =
 \begin{cases}
\rphisx{r}{s+n}{a_1,a_2,\ldots,a_r \\ b_1,b_2,\ldots,b_s, \underbrace{0,\ldots,0}_n}{z} & \text{ if $n> 0$,} \\
\rphisx{r}{s}{a_1,a_2,\ldots,a_r \\ b_1,b_2,\ldots,b_s}{z} & \text{ if $n= 0$,} \\
\rphisx{r-n}{s}{a_1,a_2,\ldots,a_r,\overbrace{0,\ldots,0}^{-n} \\ b_1,b_2,\ldots,b_s}{z} & \text{ if $n<0$.}
\end{cases}
\end{gather*}
In the case $n=0$ we will of course in general omit the $(0)$, as we then
re-obtain the usual def\/inition of ${}_r\phi_s$. Moreover, when
considering specif\/ic series, we will often omit the $r$ and $s$ from the
notation as they can now be derived by counting the number of parameters.
We also extend the def\/inition of very-well-poised series in this way:
\[
\rWsn{r}{r-1}{n}(a;b_1,\ldots,b_{r-3};q,z) =
\rphisxk{r}{r-1}{n}{a, \pm q \sqrt{a}, b_1,\ldots,b_{r-3} \\
\pm \sqrt{a},aq/b_1,\ldots,aq/b_{r-3}}{z}.
\]
Note, however, that this function cannot be obtained simply by setting some
parameters to 0 in the usual very-well-poised series. Indeed, setting the
parameter $b$ to zero in a very-well-poised series causes the corresponding
parameter $aq/b$ to become inf\/inite, making the limit fail.
For the basic hypergeometric bilateral series we use the usual notation
\[
{}_r\psi_r\left( \begin{array}{c} a_1, \ldots, a_r \\ b_1, \ldots b_r \end{array};q,z \right) =
\sum_{k\in \mathbb{Z}} \frac{(a_1, \ldots, a_r;q)_k}{(b_1, \ldots, b_r;q)_k} z^k.
\]

We def\/ine $p,q$-shifted factorials by setting
\[
(z;p,q) = \prod_{j,k\geq 0} (1-p^j q^k z).
\]
The elliptic gamma function \cite{ellgamma} is def\/ined by
\[
\Gamma(z) = \Gamma(z;p,q)= \frac{(pq/z;p,q)}{(z;p,q)} = \prod_{j,k=0}^\infty \frac{1-p^{j+1}q^{k+1}/z}{1-p^j q^kz} .
\]
We omit the $p$ and $q$ dependence whenever this does not cause confusion.
Note that the elliptic gamma function satisf\/ies the dif\/ference equations
\begin{equation}\label{eqellgammadiff}
\Gamma(qz) = \theta(z;p) \Gamma(z), \qquad \Gamma(pz) = \theta(z;q) \Gamma(z)
\end{equation}
and the ref\/lection equation
\[
\Gamma(z) \Gamma(pq/z)=1.
\]

\section{Elliptic beta integrals}\label{secellbeta}
In this section we introduce the elliptic beta integrals and we recall
their relevant properties. As a generalization of Euler's beta integral
evaluation, the elliptic beta integral was introduced by Spiridonov in
\cite{evalform}. An extension by two more parameters was shown to satisfy a
transformation formula \cite{Spirthi,e7trafo}, corresponding to a symmetry
with respect to the Weyl group of $E_7$. We can generalize the beta
integral by adding even more parameters, but unfortunately not much is
known about these integrals, beyond some quadratic transformation formulas
for $m=2$ \cite{eli} and a transformation to a multivariate integral~\cite{e7trafo}.

\begin{Definition}
Let $m \in \mathbb{Z}_{\geq 0}$. Def\/ine the set $\mathcal{H}_m =\{ z\in
\mathbb{C}^{2m+6} ~|~ \prod_{i} z_i = (pq)^{m+1} \}/ \sim$, where $\sim$ is
the equivalence relation induced by $z\sim -z$. For parameters $t\in
\mathcal{H}_m$ we def\/ine the renormalized elliptic beta integral by
\begin{equation}\label{eqdefem}
E^m(t) =\Bigl(\prod_{0\leq r<s\leq 2m+5} (t_rt_s;p,q)\Bigr)
 \frac{(p;p) (q;q)}{2} \int_{\mathcal{C}} \frac{\prod\limits_{r=0}^{2m+5} \Gamma(t_r z^{\pm 1})}{\Gamma(z^{\pm 2})} \frac{dz}{2\pi i z},
\end{equation}
where the integration contour $\mathcal{C}$ circles once around the origin
in the positive direction and separates the poles at $z = t_r p^j q^k$
($0\leq r\leq 2m+5$ and $j,k\in \mathbb{Z}_{\geq 0}$) from the poles at
$z=t_r^{-1} p^{-j}q^{-k}$ ($0\leq r\leq 2m+5$ and $j,k \in \mathbb{Z}_{\geq
 0}$). For parameters $t$ for which such a contour does not exist
(i.e.\ if $t_rt_s \in p^{\mathbb{Z}_{\leq 0}} q^{\mathbb{Z}_{\leq 0}}$) we
def\/ine $E^m$ to be the analytic continuation of the function to these
parameters.
\end{Definition}

Observe that this function is well-def\/ined, in the sense that $E^m(t) =
E^m(-t)$ by a change of integration variable $z\to -z$. We can choose the
contour in \eqref{eqdefem} to be the unit circle itself whenever $|t_r|<1$
for all $r$. If $t_rt_s=p^{-n_1}q^{-n_2}$ for some $n_1,n_2\ge 0$, $r\ne
s$, then the desired contour fails to exist, but we can obtain the analytic
continuation by picking up residues of of\/fending poles before specializing
the parameter $t$. In particular the prefactor $\prod_{0\leq r<s\leq 2m+5}
(t_rt_s;p,q)$ cancels all the poles of these residues and thus ensures
$E^m$ is analytic at those points. In this case the integral reduces to a
f\/inite sum. Indeed for $t_0t_1=p^{-n_1}q^{-n_2}$, we have
\begin{gather*}
E^m(t)
= (pq/t_0t_1;p,q)\Biggl(\prod_{\substack{0\leq r<s\leq 2m+5 \\ (r,s)\neq
 (0,1)}} (t_rt_s;p,q)\Biggr)
\Gamma(pqt_0^{ 2}, t_1/t_0) \prod_{r=2}^{2m+5} \Gamma(t_r t_0^{\pm 1})
 \\
 \phantom{E^m(t) =}{} \times
 \sum_{k=0}^{n_1}
 \prod_{r=0}^{2m+5} \frac{\theta(t_rt_0;q,p)_k }{\theta(pqt_0/t_r;q,p)_k }
\frac{ \theta(pqt_0^{2};q,p)_{2k}}{\theta(t_0^2;q,p)_{2k} }
\sum_{l= 0}^{n_2} \prod_{r=0}^{2m+5} \frac{\theta(t_rt_0;p,q)_l}{\theta(pqt_0/t_r;p,q)_l}
\frac{ \theta(pqt_0^{2};p,q)_{2l}}{\theta(t_0^2;p,q)_{2l}},
\end{gather*}
where we use the notation $\theta(x;q,p)_k=\prod_{r=0}^{k-1} \theta(x p^r;q)$.
There are other singular cases, more dif\/f\/icult to evaluate, but in general
$E^m(t)$ is analytic on all of ${\mathcal H}_m$, as follows from
\cite[Lem\-ma~10.4]{e7trafo}.

The elliptic beta integral evaluation of \cite{evalform} is now given by
\begin{Theorem}
For $t\in \mathcal{H}_0$ we have %\marginpar{Change product to product over r and s!}
\begin{equation}\label{eqeval}
E^0(t) = \prod_{0\leq r<s\leq 5} (pq/t_rt_s;p,q).
\end{equation}
\end{Theorem}
Apart from in \cite{evalform}, elementary proofs of this theorem are given
in~\cite{Spirspeb} and~\cite{e7trafo}. Moreover in~\cite{e7trafo} several
multivariate extensions of this result are presented.

A second important result is the $E_7$ symmetry satisf\/ied by $E^1$. Before we can state this in a~theorem we f\/irst have to
introduce the Weyl groups and their actions.
\begin{Definition}\label{defe8}
Let $\rho\in \mathbb{R}^8$ be the vector $\rho = (1/2,\ldots, 1/2)$.
Def\/ine the root system $R(E_8)$ of $E_8$ by %\marginpar{Change $\mathbb{Z}$ into $\mathbb{Z}^8$.}
$R(E_8) = \{ v \in \mathbb{Z}^8 \cup (\mathbb{Z}^8+\rho) ~|~ v\cdot v =2\}$.
Moreover the root system $R(E_7)$ of $E_7$ is given by $R(E_7)=\{ v\in R(E_8) ~|~ v\cdot \rho = 0\}$.
Denote by $s_{\alpha}$ the ref\/lection in the hyperplane orthogonal to $\alpha$ (i.e.\ $s_{\alpha}(\beta) = \beta - (\alpha \cdot \beta) \alpha$ for $\alpha \in R(E_8)$).
The corresponding Weyl group $W(E_7)$ is the ref\/lection group generated by $\{ s_\alpha ~|~ \alpha \in R(E_7)\}$.
Apart from the natural action of $E_7$ on $\mathbb{R}^8$, we need the action on $\mathcal{H}_1$ given by
$wt = \exp (w (\log(t)))$ for $t\in \mathcal{H}_1$ (where $\log((t_0,\ldots, t_7)) = (\log(t_0), \ldots, \log(t_7))$ and similarly
for $\exp$). Finally we will often meet the $W(E_7)$ orbit $S$ in $R(E_8)$ given by $S=\{v\in R(E_8) ~|~ s\cdot \rho =1\}$.
\end{Definition}
Note that the action of $W(E_7)$ on $\mathcal{H}_1$ is well-def\/ined due to
the equivalence of $t\sim -t$. Indeed, if we ref\/lect in a root of the form
$\rho -e_i-e_j-e_k-e_l$ then we have to take square roots of the~$t_j$, but
if we do this consistently (such that $\prod_j \sqrt{t_j} = pq$), the f\/inal
result will dif\/fer at most by a~factor~$-1$. A more thorough analysis of
this action is given in \cite{vdBRS}.

Now we can formulate the following theorem describing the transformations
satisf\/ied by $E^1$ (see \cite{Spirthi} and \cite{e7trafo}, the latter
containing also a multivariate extension).

\begin{Theorem}
The integral $E^1$ is invariant under the action of $W(E_7)$, i.e.\ for all
$w\in W(E_7)$ and $t\in \mathcal{H}_1$ we have $E^1(t) = E^1(wt)$.
\end{Theorem}

In the cited references the transformation has certain products of elliptic
gamma functions on one or both sides of the equation, but these factors are
precisely canceled by our choice of prefactor.

Let us recall the following contiguous relations satisf\/ied by $E^1$
\cite{evalform} (it is shown there for $m=0$, but the proof is identical to
that of the $m=1$ case, apart from the use of the Weyl group action). We
have rewritten it in a clearly $W(E_7)$ invariant form.
\begin{Theorem}\label{thcontell}
Let us denote $t^{\rho}= \prod_j t_j^{\rho_j}$, and $t\cdot p^{\rho} =
(t_0p^{\rho_0}, \ldots, t_7 p^{\rho_7})$. Then if $\alpha, \beta, \gamma
\in R(E_7)$ form an equilateral triangle $($i.e.\ $\alpha\cdot \beta = \alpha
\cdot \gamma = \beta\cdot \gamma=1)$ we have
\begin{gather}
\prod_{\substack{\delta\in S \\ \delta \cdot (\alpha-\beta)=\delta \cdot (\alpha-\gamma) = 1}} (t^{\delta} p^{\delta \cdot \beta};q)
t^{\gamma} \theta(t^{\beta-\gamma};q) E^1(t\cdot p^{\alpha}) \nonumber\\
\qquad\qquad {} +\prod_{\substack{\delta\in S \\ \delta \cdot (\beta-\gamma) =\delta \cdot (\beta-\alpha)=1 }} (t^{\delta} p^{\delta \cdot \gamma};q)
t^{\alpha} \theta(t^{\gamma-\alpha};q) E^1(t\cdot p^{\beta}) \nonumber\\
\qquad\qquad{} + \prod_{\substack{\delta\in S \\ \delta \cdot (\gamma-\alpha) = \delta \cdot (\gamma-\beta) = 1}} (t^{\delta} p^{\delta \cdot \alpha};q)
t^{\beta} \theta(t^{\alpha-\beta};q) E^1(t\cdot p^{\gamma})=0 .\label{eqpcont}
\end{gather}
\end{Theorem}
\begin{proof}
Observe that the relation is satisf\/ied by the integrands when
$\alpha=e_1-e_0$, $\beta=e_2-e_0$ and $\gamma=e_3-e_0$, where $\{e_i\}$
form the standard orthonormal basis of $\mathbb{R}^8$, due to the
fundamental relation
\begin{equation}\label{eqtheta}
\frac{1}{y} \theta\big(wx^{\pm1}, yz^{\pm 1};q\big) +
\frac{1}{z} \theta\big(wy^{\pm1}, zx^{\pm 1};q\big) +
\frac{1}{x} \theta\big(wz^{\pm1}, xy^{\pm 1};q\big) =0.
\end{equation}
Integrating the identity now proves the contiguous relations for these
special $\alpha$, $\beta$ and $\gamma$. As the equation is invariant under
the action of $W(E_7)$, which acts transitively on the set of all
equilateral triangles of roots, the result holds for all such triangles.
\end{proof}

These contiguous relations can be combined to obtain relations of three
$E^1$'s which dif\/fer by shifts along any vector in the root lattice of
$E_7$ (i.e., the smallest 7-dimensional lattice in~$\mathbb{R}^8$ containing~$R(E_7)$). In particular the equation relating $E^1(t \cdot p^{\alpha})$,
$E^1(t)$ and $E^1(t \cdot p^{-\alpha})$ for $\alpha = e_1-e_0$ is the
elliptic hypergeometric equation studied by Spiridonov in, amongst others,~\cite{Spircehf}.

\section{Limits to basic hypergeometric functions}\label{seclimits}

In order to obtain basic hypergeometric limits from these integrals we let
$p\to 0$. As our parameters can not be chosen independently of $p$ (due to
the balancing condition), we have to explicitly describe how they behave as
$p\to 0$. Dif\/ferent ways the parameters depend on $p$ require dif\/ferent
ways of obtaining the limit. In this section we describe the dif\/ferent
limits of interest to us.

Using the notation of Theorem \ref{thcontell} we see that $u \cdot
p^{\alpha}$, for $u$ independent of $p$, is an element of~$\mathcal{H}_m$
if $\alpha \in \mathbb{R}^{2m+6}$ with $\sum_{r} \alpha_r = m+1$, and $u
\in \tilde{\mathcal{H}}_m = \{ z\in \mathbb{C}^{2m+6} ~|~ \prod_{i} z_i =
q^{m+1} \}/ \sim$ (where we again have $z\sim -z$). In particular in this
section we will describe various conditions on $\alpha$ which ensure that
the limit
\begin{equation}\label{eqlim}
B^m_{\alpha}(u) = \lim_{p\to 0} E^m(u \cdot p^{\alpha})
\end{equation}
is well-def\/ined, and give explicit expressions for this limit. In
particular, for $m=1$ we would like such expressions for $\alpha$ in the
entire Hesse polytope as def\/ined in Theorem \ref{th1}.

The simplest way to obtain a limit is given by the following proposition.

\begin{Proposition}\label{proptrivlimit}
For $\alpha \in \mathbb{R}^{2m+6}$ satisfying $\sum_r \alpha_r = m+1$ and
such that $0\leq \alpha_r \leq 1$ for all $r$, the limit in \eqref{eqlim}
exists and we have
\[
B^m_\alpha(u) = \prod_{\substack{0\leq r<s\leq 2m+5 \\ \alpha_r=\alpha_s=0}} (u_ru_s;q)
 \frac{(q;q)}{2} \int_{\mathcal{C}} (z^{\pm 2};q)
 \frac{\prod\limits_{r\colon \alpha_r = 1} (q/u_r z^{\pm 1};q)}{\prod\limits_{r\colon \alpha_r=0} (u_r z^{\pm 1};q)} \frac{dz}{2\pi i z},
\]
where the contour is a deformation of the unit circle which separates the
poles at $z=u_r q^n$ ($\alpha_r=0, n\geq 0$) from those at $u_r^{-1}
q^{-n}$ $(\alpha_r=0$, $n\geq 0)$.
\end{Proposition}
We want to stress that the limit also exists if the integral above is not well-def\/ined
(i.e.\ when there exists no proper contour, when $u_ru_s=q^{-n}$ for some $\alpha_r=\alpha_s=0$).
In that case the limit~$B^m_{\alpha}$ is equal to the
analytic continuation of the integral representation to these values of the parameters.
\begin{proof}
Observe that we can determine limits of the elliptic gamma function by
\[
\lim_{p\to 0} \Gamma(p^{\gamma}z) = \begin{cases}
\frac{1}{(z;q)} & \text{if $\gamma=0$}, \\
1 & \text{if $0<\gamma<1$}, \\
(q/z;q) & \text{if $\gamma=1$}.
\end{cases}
\]
In fact $\Gamma(p^{\gamma} z)$ is well-def\/ined and continuous in $p$ at $p=0$ for $0\leq \gamma\leq 1$. These
limits can thus be obtained by just plugging in $p=0$.
Similarly observe that
\[
\lim_{p\to 0} (p^{\gamma} z;p,q) = \begin{cases}
(z;p,q) & \text{if $\gamma=0$}, \\
1 & \text{if $\gamma>0$}.
\end{cases}
\]
The result now follows from noting that an integration contour which
separates the poles at $z=u_r q^n$ ($\alpha_r=0, n\geq 0$) from those at
$u_r^{-1} q^{-n}$ ($\alpha_r=0, n\geq 0$) will also work in the def\/inition
of $E_m(u \cdot p^{\alpha})$ if $p$ is small enough (as the poles of the
integrand created by $u_r$'s with $\alpha_r >0$ will all converge either to
0 or to inf\/inity; in particular they will remain on the correct side of the
contour for small enough $p$). Thus we can just plug in $p=0$ in the
integral to obtain the limit.

This proof only works when the parameters $u$ are such that there exists a
contour for the limiting integral. However, this implies these limits work
outside a f\/inite set of co-dimension one divisors. Indeed, on compacta
outside these divisors the convergence is uniform.
Using the Stieltjes--Vitali theorem we can conclude that the limit also holds on these
divisors, and is in fact uniform on compacta of the entire parameter space. Moreover
Stieltjes--Vitali tells us that the limit function is analytic in these points as well.
\end{proof}

A second kind of limit, following \cite[\S~5]{limits}, can be obtained by
f\/irst breaking the symmetry of the integrand. This leads to the following
proposition.

\begin{Proposition}\label{propsymbreak}
Let $\alpha \in \mathbb{R}^{2m+6}$ satisfy $\sum_r \alpha_r = m+1$ and
$\alpha_0\leq \alpha_1 \leq \alpha_2$. Define
$\beta=\alpha_0+\alpha_1+\alpha_2$ and impose the extra conditions
$\beta\leq \alpha_r\leq -\beta$ for $r=0,1,2$ and $-\beta\leq \alpha_r\leq
1+\beta$ for $r\geq 3$. Then the limit in \eqref{eqlim} exists, and takes
one of the following forms:
\begin{itemize}\itemsep=0pt
\item
If $\alpha_0=\alpha_1=-\alpha_2$ (thus $\beta=\alpha_0$), then
\begin{gather*}
B_\alpha^m(t) =
\frac{\prod\limits_{r\geq 3\colon \alpha_r=-\alpha_0} (u_ru_0,u_ru_1;q) }{(q/u_0u_2,q/u_1u_2;q)}
(u_0u_1;q)^{1_{\{\alpha_0=-1/2\}}}
 \\
 \phantom{B_\alpha^m(t) =}{} \times (q;q) \int_{\mathcal{C}} \theta(u_0u_1u_2/z;q)
 \frac{(q/u_2z;q)}{(u_0/z,u_1/z;q)} \\
\phantom{B_\alpha^m(t) =}{}
 \times
 \frac{\prod\limits_{r\geq3 \colon \alpha_r=1+\alpha_0} (qz/u_r;q)}{\prod\limits_{r\geq 3\colon \alpha_r=-\alpha_0} (u_rz;q)}
 \left( \frac{(1-z^2) (qz/u_2;q) }{(u_0z,u_1z;q)} \right)^{1_{\{\alpha_0=-1/2\}}}
 \frac{dz}{2\pi i z},
\end{gather*}
where the contour separates the downward from the upward pole
sequences. Here $1_{\{\alpha_0=-1/2\}}$ equals 1 if $\alpha_0=-1/2$ and 0
otherwise.
\item
If $\alpha_0<\alpha_1=-\alpha_2$ (again $\beta=\alpha_0$), then
\begin{gather*}
B_\alpha^m(u) =
\frac{(q;q)}{(q/u_1u_2;q)}
 \prod_{\substack{3\leq r\leq 2m+5 \\ \alpha_r=-\alpha_0}} (u_ru_0;q)
\int_{\mathcal{C}} \theta(u_0u_1u_2/z;q) \\
\phantom{B_\alpha^m(u) =}{} \times \frac{1}{(u_0/z;q)}
\frac{\prod\limits_{r\geq 3\colon \alpha_r=1+\alpha_0} (qz/u_r;q) }{
\prod\limits_{r\geq 3\colon \alpha_r=-\alpha_0} (u_rz;q)}
\left(\frac{(1-z^2)}{(u_0z;q)}\right)^{1_{\{\alpha_0=-1/2}\}} \frac{dz}{2\pi i z},
\end{gather*}
where the contour separates the downward poles from the upward ones.
\item
Finally, if $\alpha_1<-\alpha_2$ (thus $\beta<\alpha_0$), then
\begin{equation*}
B^m(t) =
 (q;q) \int_{\mathcal{C}}
 \theta(u_0u_1u_2/z;q)
 \frac{\prod\limits_{r\colon \alpha_r=1+\beta} (qz/u_r;q)}{\prod\limits_{r\colon \alpha_r=-\beta} (u_rz;q)}
 \big(1-z^2\big)^{1_{\{\beta=-1/2\}}}
 \frac{dz}{2\pi i z},
\end{equation*}
where the contour excludes the poles but circles the essential singularity
at zero.
\end{itemize}
\end{Proposition}

\begin{proof}
In order to obtain these limits we will break the symmetry of the
integral. We f\/irst rewrite \eqref{eqtheta} in the form
\[
\frac{\theta(s_0s_1s_2/z,s_0z,s_1z,s_2z;q)}{\theta(z^2,s_0s_1,s_0s_2,s_1s_2;q)} +
\big(z\leftrightarrow z^{-1}\big) = 1.
\]
Since the integrand of $E^m$ is invariant under the interchange of $z\to
z^{-1}$, we can multiply by the left hand side of the above equation and
observe that the integrand splits in two parts, each integrating to the
same value. Therefore, the integral itself is equal to twice the integral
of either part, and we thus obtain
\begin{gather}
E^m(t) =\prod_{0\leq r<s\leq 2m+5} (t_rt_s;p,q)
 (p;p) (q;q)\nonumber \\
 \phantom{E^m(t) =}{} \times \int_{\mathcal{C}} \frac{\prod\limits_{r=0}^{2m+5} \Gamma(t_r z^{\pm 1})}{\Gamma(z^{\pm 2})} \frac{\theta(s_0s_1s_2/z,s_0z,s_1z,s_2z;q)}{\theta(z^2,s_0s_1,s_0s_2,s_1s_2;q)}
 \frac{dz}{2\pi i z}.\label{eqemasym}
\end{gather}
The poles introduced by the factor $1/\theta(z^2;q)$ are canceled by zeros
of the factor $1/\Gamma(z^{\pm 2})$, as we have
\[
\frac{1}{\Gamma(z^{\pm 2}) \theta(z^2;q)} = \frac{\Gamma(pq z^2)}{\Gamma(pz^2)} = \theta\big(pz^2;p\big)=\theta\big(z^{-2};p\big)
\]
using the dif\/ference and ref\/lection equations satisf\/ied by the elliptic
gamma functions. This process therefore does not introduce any extra poles
to the integrand; we may therefore use the same contour as before. In
fact, since some of the original poles might have been cancelled, the constraints
on the contour can be correspondingly weakened.

Now, specialize $s_r=t_r$ ($r=0,1,2$) in \eqref{eqemasym} and simplify to
obtain
\begin{gather}
E^m(t) =
\frac{\prod\limits_{0\leq r<s\leq 2} (pt_rt_s;p,q) \prod\limits_{r=0}^2 \prod\limits_{s=3}^{2m+5} (t_rt_s;p,q) \prod\limits_{3\leq r<s\leq 2m+5} (t_rt_s;p,q)}{(q/t_0t_1,q/t_0t_2,q/t_1t_2;q)}
\nonumber \\
\phantom{E^m(t) =}{} \times (p;p) (q;q)\int_{\mathcal{C}}
 \theta(z^{-2};p) \theta(t_0t_1t_2/z;q)
 \prod_{r=0}^2 \Gamma(pt_rz,t_r/z)
 \prod_{r=3}^{2m+5} \Gamma(t_r z^{\pm 1})
 \frac{dz}{2\pi i z} .\label{eqsymbreak3spec}
\end{gather}
Now change the integration variable $z\to zp^{\beta}$.
The inequalities $\alpha_0,\alpha_1,\alpha_2\ge\beta$ and
$-\beta\le \alpha_r$, $3\le r$ ensure that the downward poles remain
bounded and the upward poles remain bounded away from 0 as $p\to 0$. There
thus (for generic $u_r$) exists a contour valid for all suf\/f\/iciently small
$p$. After f\/ixing such a contour, the limit again follows by simply
plugging in $p=0$; the constraints on $\alpha$ are necessary and suf\/f\/icient
to ensure that all gamma functions in the integrand have well-def\/ined
limits.
\end{proof}

The two previous limits still do not allow us to take limits for each
possible vector in the Hesse polytope (in the $m=1$ case). Indeed (as we
will show below) we have covered the polytope, modulo the action of $S_8$
to sort the entries $\alpha_0\leq \cdots \leq \alpha_{7}$, as long as
either $\alpha_0\geq 0$ (Proposition~\ref{proptrivlimit}) or
$\alpha_1+\alpha_2\leq 0$ (Proposition~\ref{propsymbreak}). The remaining
limits require a more careful look and are given by the following
proposition.
\begin{Proposition}\label{propsum}
Let $\alpha \in \mathbb{R}^{2m+6}$ satisfy $\sum_r \alpha_r = m+1$ and
assume $-1/2 \leq \alpha_0<0$, $\alpha_0\leq \alpha_1\leq
\alpha_2\leq \cdots\leq \alpha_{m+3}\le 1+\alpha_0$ and
for $2\leq k\leq m+3$,
\begin{equation*}%\label{eqarbs}
\sum_{r\in I} (\alpha_r+\alpha_0) \geq 2\alpha_0 , \qquad I \subset \{1,2,\ldots, 2m+5\}, \qquad |I|=k
\end{equation*}
hold. Then the limit in \eqref{eqlim} exists.
\begin{itemize}\itemsep=0pt
\item
If $\alpha_0=\alpha_1=-1/2$ $($thus $\alpha_2=\cdots=\alpha_{2m+5}=1/2)$ we have
\begin{gather*}
B^m_\alpha(u) = \frac{\prod\limits_{r=2}^{2m+5} (u_ru_1,qu_0/u_r;q)}{ (qu_0^2,u_0u_1,u_1/u_0;q)} \\
\phantom{B^m_\alpha(u) =}\times
\rWs{2m+8}{2m+7}\big(u_0^2;u_0u_1,u_0u_2,\ldots,u_0u_{2m+5};q,q\big)
+ (u_0 \leftrightarrow u_1).
\end{gather*}
\item
If $\alpha_0=-1/2< \alpha_1$, and if $\alpha_1+\alpha_2=0$ the extra condition $|u_1u_2|<1$, we have with $n=\#\{r\colon \alpha_r<1/2\}-3 $
\begin{gather*}
B^m_\alpha(u) = \frac{\{(u_1u_2;q)\}\!\!\prod\limits_{r\colon \alpha_r=1/2}\!\! (qu_0/u_r;q)}{(qu_0^2;q)}
\rWsn{}{}{n}\Bigg(\!u_0^2;u_0u_r\colon \alpha_r=1/2 ;q, u_0^n \!\!\prod_{r>0\colon \alpha_r<1/2}\!\! u_r\!\Bigg),
\end{gather*}
where the notation implies we take as parameters $u_0u_r$ for those $r$
which satisfy $\alpha_r=1/2$ and the factor $(u_1u_2;q)$ appears only if $\alpha_1+\alpha_2=0$.

\item If $-1/2<\alpha_0=\alpha_1<0$ then
\begin{gather*}
B^m_\alpha(u) = \frac{\prod\limits_{\alpha_r=-\alpha_0} (u_1u_r;q) \prod\limits_{\alpha_r=1+\alpha_0} (qu_0/u_r;q)}{(u_1/u_0;q)}
\\
\phantom{B^m_\alpha(u) =}{} \times
\rphisxkn{n}{\quad \;\, \qquad u_0u_r\colon \alpha_r=-\alpha_0 \\ qu_0/u_1,qu_0/u_r\colon \alpha_r=1+\alpha_0 }{q}
+ (u_0 \leftrightarrow u_1),
\end{gather*}
where $n= \# \{r\colon \alpha_r=-\alpha_0\} - \# \{r\colon \alpha_r=1+\alpha_0\} -2$.

\item
If $-1 < 2\alpha_0 = \sum_{r\geq 1\colon \alpha_r+\alpha_0 < 0} (\alpha_r+\alpha_0)$ and $\alpha_1>\alpha_0$,
and if $\alpha_1+\alpha_2=0$ the extra condition $|u_1u_2|<1$, we get
\begin{gather*}
B^m_\alpha(u) = \{(u_1u_2;q)\} \prod_{r\colon \alpha_r=1+\alpha_0} (qu_0/u_r;q) \\
\hphantom{B^m_\alpha(u) =}{}\times \rphisxkn{n}{ u_0u_r\colon \alpha_r=-\alpha_0 \\ qu_0/u_r\colon \alpha_r=1+\alpha_0}{u_0^{-2} \prod_{r>0\colon \alpha_r<-\alpha_0} (u_ru_0)},
\end{gather*}
where $n= \# \{r\colon \alpha_r<-\alpha_0\}
-4-\#\{r\colon \alpha_r=1+\alpha_0\}+\#\{r\colon \alpha_r=-\alpha_0\}$, and the factor
$(u_1u_2;q)$ appears only if $\alpha_1+\alpha_2=0$.

\item
Finally if $2\alpha_0< \sum_{r\geq 1\colon \alpha_r+\alpha_0 < 0} (\alpha_r+\alpha_0)$ we get
\[
B^m_\alpha(u) = \prod_{r\colon \alpha_r=1+\alpha_0} (qu_0/u_r;q).
\]
\end{itemize}
\end{Proposition}

\begin{proof}
Note that limits in the cases $\alpha_0=\alpha_1=-1/2$ and $-1/2<\alpha_0=\alpha_1\geq -\alpha_r$ ($r\geq 2)$
are given in Proposition \ref{propsymbreak}. Together with the limits in this proposition we have thus covered
all of the possible values for $\alpha$ at least once.

Due to the condition $\alpha_0<0$, in the integral def\/inition of
$E^m(u\cdot p^{\alpha})$ there always exist poles which have to be excluded
from the contour which go to zero as $p\to 0$, for example
$z=u_0p^{\alpha_0} q^k$ for $k\in \mathbb{Z}_{\geq 0}$. Similarly there are
poles going to inf\/inity as $p\to 0$ which have to be included. The proof
of this proposition in essence consists of f\/irst picking up the residues
belonging to these poles, and taking the contour of the remaining integral
close to the unit circle. Subsequently we take the limit as $p\to 0$ (which
involves picking up an increasing number of residues), and show that the
sums of these residues converge to one or two basic hypergeometric series,
while the remaining integral converges to zero.

\looseness=-1 Proving that we are allowed to interchange sum and limit and that the
remaining integral vanishes in the limit consists of a calculation giving
upper bounds on the integrand and residues, after which we can use
dominated convergence. This calculation is quite tedious and hence omitted.

The necessary bounds of the elliptic gamma function can be obtained by using the
dif\/ference equation \eqref{eqellgammadiff} to ensure the argument of the
elliptic gamma function is of the form $\Gamma(p^{\gamma} z)$ for $0\leq \gamma \leq 1$, and
using the known asymptotic behavior of the theta functions outside their poles and zeros.

This gives a bound on the integrand for a contour which is at least
$\epsilon>0$ away from any poles of the integrand, and moreover gives us a
summable bound on the residues, thus showing that any residues
corresponding to points not of the form $z=t_0q^n$ must vanish in the limit
(here we use $\alpha_0<\alpha_r$ for $r>0$). However a contour as required
does in general not exist for all values of~$p$.

Therefore choose parameters $u$ in a compact subset $K$ of the complement
of the $p$-independent divisors (i.e.\ such that there are no
$p$-independent pole-collisions of the integrand of $E^m$). For any $p$
for which we can obtain a contour which stays $\epsilon$ away from any
poles of the integrand (for all $u\in K$), we can use our estimates to
bound $|E^m-B^m_{\alpha}|$ uniformly for $u\in K$ and $a=|p|$, with the
bound going to zero as $a\to 0$. As long as $\log(p)$ stays $\epsilon$
away from conditions of the form
$u_r^{-1}u_s^{-1}q^{-n}=p^{l+\alpha_r+\alpha_s}$ ($l,n\in \mathbb{N}$,
$u_r,u_s$ range over the projection of $K$ to the $r$'th and $s$'th
coordinate) the poles of the integrand near the unit circle stay
$\mathcal{O}(\epsilon)$ away from each other and we can f\/ind a desired
contour. Moreover this ensures that the residues we pick up are at least~$\epsilon$ distance away from any other poles.

Note that we only need to consider conditions with $l+\alpha_r+\alpha_s<0$ as the other condition cannot be
satisf\/ied for small enough $p$, this implies there is only a f\/inite set of possible $l$, $r$ and~$s$.
Hence, if we start with small enough $K$ and $\epsilon$, we can ensure that
these excluded values of $p$ form disjoint sets. In particular we can, in
the $p$-plane, create a circle around these disjoint sets, and use the
maximum principle to show that $E^m-B^m_\alpha$ is bounded in absolute
value inside these circles by the maximum of the absolute value on the
circle. As the circle consists entirely of $p$'s for which our estimates
work, we see that inside the circle the dif\/ference is bounded as well (by a
bound corresponding to a slightly larger radius). Hence for all values of
$p$ with $|p|>0$ we f\/ind that $|E^m-B^m_{\alpha}|$ is bounded uniformly in
$u$ and $a=|p|$ with the bound going to zero as $a\to 0$. in particular the
limit holds uniformly for $u\in K$. Finally we can use the
Stieltjes--Vitali theorem again to show the limit holds for all values of
$u$.
\end{proof}

Note that there is some overlap in the conditions of Proposition \ref{propsymbreak} and
Proposition \ref{propsum}. Indeed we get two dif\/ferent representations of the same function (one integral
and one series) in the case of $\alpha \in \mathbb{R}^{2m+6}$ satisfying $\sum_r \alpha_r = m+1$,
$\alpha_0 \leq \alpha_r\leq -\alpha_0$ for $r=1,2$, $\alpha_1+\alpha_2=0$,
$-\alpha_0 \leq \alpha_r \leq 1+\alpha_0$ for $r\geq 3$.

Moreover, in some special cases we have integral representations of the series in Proposition~\ref{propsum}, %which are not contain. Moreover, in some cases we have similar integral representations for limits
which were not covered in Proposition~\ref{propsymbreak}. Moreover we sometimes f\/ind a second, slightly dif\/ferent, expression
for the integrals of Proposition~\ref{propsymbreak}. Indeed we have
\begin{Proposition}\label{propsumintegral}
For $\alpha \in \mathbb{R}^{2m+6}$ satisfying $\sum_r \alpha_r = m+1$ and $\alpha_0\leq \alpha_1 \leq \cdots \leq \alpha_{2m+5}$
such that $-1/2\leq \alpha_0=\alpha_1 < 0$ and $-\alpha_0 \leq \alpha_2$ and $\alpha_{2m+5}\leq 1+\alpha_0$
the limit in \eqref{eqlim} exists and we have
\begin{gather*}
B^m_\alpha(u) = \prod_{r\geq 2\colon \alpha_r=-\alpha_0} (u_0u_r,u_1u_r;q)
 (q;q) \int_{\mathcal{C}} \frac{\theta(u_0u_1w/z,wz;q)}{\theta(u_0w,u_1w;q)}
\\
\phantom{B^m_\alpha(u) =}{} \times
 \frac{\prod\limits_{r\geq 2\colon \alpha_r=1+\alpha_0} (qz/u_r;q) }{\prod\limits_{r\geq 2\colon \alpha_r=-\alpha_0} (u_rz;q)}
 \frac{1}{(u_0/z,u_1/z;q)}
 \left( \frac{1-z^2}{(u_0z,u_1z;q)} \right)^{1_{\{\alpha_0=-1/2}\}}
 \frac{dz}{2\pi iz};
\end{gather*}
where the contour is a deformation of the unit circle separating the poles in downward sequences from the poles in upward sequences.
\end{Proposition}
The theta functions involving the extra parameter $w$ combine to give a
$q$-elliptic function of $w$ and in fact the integrals are independent of
$w$ (though this is only obvious from the fact that~$B_{\alpha}(u)$ does
not depend on $w$). In the case $\alpha_0=\alpha_1=-\alpha_2$, which is
also treated in Proposition~\ref{propsymbreak}, we can specialize $w=u_2$
to re-obtain the previous integral expression of that limit.

\begin{proof}
As in the proof of Proposition \ref{propsymbreak}, we start with the
symmetry broken version of $E^m$, as in \eqref{eqemasym}. Now we
specialize $s_0=t_0$, $s_1=t_1$ and $s_2=w$. Thus we get
\begin{gather}
E^m(t) = \frac{(pt_0t_1;p,q)}{(q/t_0t_1;q)} \prod_{r=0,1} \prod_{s=2}^{2m+5} (t_rt_s;p,q)
\prod_{2\leq r<s\leq 2m+5} (t_rt_s;p,q)
 (p;p) (q;q) \nonumber\\
\phantom{E^m(t) =}{}
 \times \int_{\mathcal{C}}
 \prod_{r=0}^1 \Gamma(pt_rz,t_r/z)
 \prod_{r=2}^{2m+5} \Gamma(t_r z^{\pm 1}) \frac{\theta(t_0t_1w/z,wz;q)}{\theta(t_0w,t_1w;q)} \theta\big(z^{-2};p\big)
 \frac{dz}{2\pi i z}.\label{eqsymbreak2spec}
\end{gather}
Replacing $z\to p^{\alpha_0}z$ and $w\to p^{-\alpha_0}w$ and using
$t_r=p^{\alpha_r}u_r$ we can subsequently plug in $p=0$ as before to obtain
the desired limit.
\end{proof}

\section{The polytopes}\label{secpolytope}
In this section we describe a polytope (for each value of $m$) such that
points of the polytope correspond to vectors $\alpha$ with respect to which
we can take limits. Moreover we describe how the limiting functions
$B_{\alpha}$ depend on geometrical properties of $\alpha$ in the polytope.

Let us begin by def\/ining the polytopes.
\begin{Definition}\label{defPm}
For $m\in \mathbb{N}$ we def\/ine the vectors $\rho^{(m)}$, $v_{j_1j_2 \cdots j_m}^{(m)}$ ($0\leq j_1<j_2<\cdots<j_m\leq 2m+5$) and $w_{ij}^{(m)}$ ($0\leq i<j\leq 2m+5$) by
\begin{equation*}
\rho^{(m)} = \frac12 \sum_{r=0}^{2m+5} e_r, \qquad
v_{j_1j_2 \cdots j_{m+1}}^{(m)} = \sum_{r=1}^{m+1} e_{j_r}, \qquad
w_{ij}^{(m)} = \rho^{(m)} - e_i-e_j,
\end{equation*}
where the $e_k$ ($0\leq k\leq 2m+5$) form the standard orthonormal basis of $\mathbb{R}^{2m+6}$. Sometimes we write $v_{S}^{(m)}$ for $S\subset \{0,1,\ldots, 2m+5\}$ with
$|S|=m+1$.

The polytope $P^{(m)}$ is now def\/ined as the convex hull of the vectors
$v_{S}^{(m)}$ ($|S|=m+1$) and $w_{ij}^{(m)}$ ($0\leq i<j\leq 2m+5$).
In the notation for both vectors and polytopes we often omit the $(m)$ if the value of $m$ is clear from context.
\end{Definition}

We will now state the main results of this section. The proofs follow after
we have stated all theorems. The main result of this section will be the
following theorem.

\begin{Theorem}\label{thmlimitsexist}
For $\alpha\in P^{(m)}$ the limit in \eqref{eqlim} exists and
$B_{\alpha}^m(u)$ depends only on the (open) face of $P^{(m)}$ which
contains $\alpha$, i.e.\ if $\alpha$ and $\beta$ are contained in the same
face of $P^{(m)}$ then $B_{\alpha}^m(u) = B_{\beta}^m(u)$.
\end{Theorem}

Next we have the following iterated limit property.
\begin{Theorem}\label{thmiteratedlimits}
Let $\alpha, \beta \in P^{(m)}$. Then the iterated limit property holds, i.e.\ \[
\lim_{x\to 0} B_{\alpha}^m(x^{\beta-\alpha} u) = B_{t\alpha+ (1-t)\beta}(u)
\]
for any $0<t<1$.
\end{Theorem}
As $t\alpha+(1-t)\beta$ is contained in the same face of $P^{(m)}$ for all
values $0<t<1$, we already know that the right hand side does not depend on
$t$.

The iterated limit property shows that all the functions associated to
faces can be obtained as limits of the (basic hypergeometric!) functions
associated to vertices of the polytope. There are only two dif\/ferent limits
associated to vertices (as there are only two dif\/ferent vertices up to
permutation symmetry), so all results follow from identities satisf\/ied by
these two functions. Indeed the idea of this article is not so much to show
new identities as it is to show how many known identities f\/it in a uniform
geometrical picture. Moreover this picture allows us to simply classify all
formulas of certain kinds.

As an immediate corollary of the iterated limit property we f\/ind the last
main theorem of this section.

\begin{Theorem}\label{thmorthface}
For $\alpha\in P^{(m)}$ the function $B_{\alpha}(u)$ depends only on the
space orthogonal to the face containing $\alpha$. To be precise if $\beta$
is in the same (open) face as $\alpha$, then
\[
B_{\alpha}(u) = B_{\alpha}(u \cdot x^{\alpha-\beta}).
\]
\end{Theorem}
\begin{proof}
Consider the line $v(t)=t\alpha + (1-t)\beta$. As $\alpha$ and $\beta$ are
in the same open face there exists $\lambda_1>1$ such that $v(\lambda_1)$ is
also in this face. Moreover $\alpha$ is a strictly convex linear
combination of $v(\lambda_1)$ and $\beta$, and $v(\lambda_1)-\beta =
\lambda_1 (\alpha-\beta)$. Now observe that
\[
B_{\alpha}(u) = \lim_{y\to 0} B_{v(\lambda_1)}(y^{v(\lambda_1)-\beta} u) =
\lim_{y\to 0} B_{v(\lambda_1)}(y^{v(\lambda_1)-\beta} x^{\frac{v(\lambda_1)-\beta}{\lambda_1}} u) =
B_{\alpha}(u \cdot x^{\beta-\alpha})
\]
by the iterated limit property. Here we replaced $y\to y x^{1/\lambda_1}$
in the second equality.
\end{proof}

To prove the f\/irst two main theorems, Theorems \ref{thmlimitsexist} and
\ref{thmiteratedlimits}, we need to split up $P^{(m)}$ in several (to be
precise $1+(2m+6)+ \binom{2m+6}{3}$, but essentially only 3) dif\/ferent
parts. Let us begin with def\/ining the smaller polytopes. Recall the
def\/inition of the vectors $\rho^{(m)}$, $v_S^{(m)}$ and $w_{ij}^{(m)}$ from
Def\/inition \ref{defPm}.
\begin{Definition}
We def\/ine the three convex polytopes $P_{\rm I}^{(m)}$, $P_{\rm II}^{(m)}$ and $P_{\rm III}^{(m)}$ by
\begin{itemize}\itemsep=0pt
\item $P_{\rm I}^{(m)}$ is the convex hull of the vectors $v_S^{(m)}$ ($S\subset \{0,1,\ldots,2m+5\}$);
\item $P_{\rm II}^{(m)}$ is the convex hull of the vectors $v_{S}^{(m)}$ ($S\subset \{1,2,\ldots,2m+5\}$) and $w_{0j}^{(m)}$ ($1\leq j\leq 2m+5$);
\item $P_{\rm III}^{(m)}$ is the convex hull of the vectors $v_S^{(m)}$ ($S\subset\!\{3,4,\ldots,2m+5\}$) and $w_{ij}^{(m)}$ \mbox{($0\leq i<j\leq 2$)}.
\end{itemize}
Here we always have $|S|=m+1$ (otherwise $v_S^{(m)}$ would not make sense).
\end{Definition}
The polytopes $P_{\rm I}^{(m)}$, $P_{\rm II}^{(m)}$ and $P_{\rm III}^{(m)}$ correspond to
limits in Propositions \ref{proptrivlimit}, \ref{propsum}, respectively~\ref{propsymbreak}.
The following proposition allows us to prove things about $P^{(m)}$ by proving them for these
simpler polytopes.

\begin{Proposition}\label{propdecomp}
Denote $\sigma(A) = \{ \sigma(a) ~|~ a\in A\}$ for some permutation $\sigma \in S_{2m+6}$. Then we have
\begin{equation}\label{eqpmdecomposition}
P^{(m)} = P_{\rm I}^{(m)} \cup \bigcup_{\sigma \in S_{2m+6}} \sigma\big(P_{\rm II}^{(m)} \big) \cup \bigcup_{\sigma \in S_{2m+6}} \sigma\big(P_{\rm III}^{(m)}\big).
\end{equation}
\end{Proposition}

\begin{proof}
It is suf\/f\/icient to show that given any set $V$ of vertices of $P^{(m)}$ their
convex hull can be written as the union of subsets of the polytopes on the
right hand side. If $V$ does not contain one of the following bad sets
\begin{enumerate}\itemsep=0pt
\item $\{w_{ij}, v_{S_1}, v_{S_2} \}$ for $i\in S_1, j\in S_2$;
\item $\{w_{ij}, v_S\}$ for $i,j\in S$;
\item $\{w_{ij}, w_{kl} \}$;
\item $\{w_{ij}, w_{ik}, v_S\}$ for $i\in S$,
\end{enumerate}
where $i$, $j$, $k$, $l$ denote dif\/ferent integers, then $V$ is contained in the
sets of vertices of $P_{\rm I}^{(m)}$ or one of the permutations of $P_{\rm II}$ or
$P_{\rm III}$. This follows from a simple case analysis depending on the number and kind
of $w_{ij}$'s in $V$.

Given any point $p$ in the (closed) convex hull ${\rm ch}(V)$ of $V$, with $p=\sum_{v\in
 V} a_v v$, we can write ${\rm ch}(V)= \bigcup_{v\colon a_v>0} {\rm ch}((V \backslash
\{v\})\cup\{p\})$. Indeed any point $q$ in ${\rm ch}(V)$ can be written as
$q=\sum_{v\in V} b_v v = \gamma p + \sum_{v\in V} (b_v - a_v \gamma) v$,
where we can take $\gamma \geq 0$ to be such that $b_{v'}=a_{v'}\gamma$ for
some $v'$ with $a_{v'}>0$ and $b_v \geq a_v \gamma$ for all $v\in V$. Now
$q$ clearly is a convex linear combination of elements of $(V\backslash\{v'\}) \cup \{p\}$.
This argument is visualized in Fig.~\ref{piccharg}.
As a generalization we obtain that if $p\in {\rm ch}(W)$ for some set $W$ we have
that ${\rm ch}(V) \subset \bigcup_{v\colon a_v>0} {\rm ch}((V\backslash \{v\}) \cup W )$.

\begin{figure}[t]
\centerline{\begin{picture}(100,100)
\put(0,0){\circle*{3}}
\put(100,0){\circle*{3}}
\put(50,87){\circle*{3}}
\put(40,40){\circle*{3}}
\drawline(0,0)(100,0)(50,87)(40,40)(0,0)(50,87)
\drawline(40,40)(100,0)
\put(-10,0){\makebox(0,0){$v_1$}}
\put(110,0){\makebox(0,0){$v_2$}}
\put(60,87){\makebox(0,0){$v_3$}}
\put(50,40){\makebox(0,0){$p$}}
\drawline(10,0)(10,10)
\drawline(20,0)(20,20)
\drawline(30,0)(30,30)
\drawline(40,0)(40,40)
\drawline(50,0)(50,33)
\drawline(60,0)(60,26)
\drawline(70,0)(70,20)
\drawline(80,0)(80,13)
\drawline(90,0)(90,6)
\drawline(5,9)(7,7)
\drawline(10,17)(14,14)
\drawline(15,26)(22,22)
\drawline(20,35)(29,29)
\drawline(25,43)(36,36)
\drawline(30,52)(41,45)
\drawline(35,61)(43,56)
\drawline(40,69)(46,66)
\drawline(45,78)(48,77)
\drawline(55,78)(47,71)
\drawline(60,69)(44,58)
\drawline(65,61)(41,45)
\drawline(70,52)(55,43)
\drawline(75,43)(54,31)
\drawline(80,35)(63,25)
\drawline(85,26)(72,19)
\drawline(90,17)(81,13)
\drawline(95,9)(90,7)
\end{picture}}
\caption{${\rm ch}(v_1,v_2,v_3) = {\rm ch}(v_1,v_2,p) \cup {\rm ch}(v_1,v_3,p) \cup {\rm ch}(v_2,v_3,p)$.}
\label{piccharg}
\end{figure}
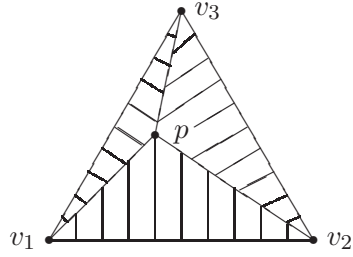

Now we can consider a set of vertices $V$ containing a bad conf\/iguration,
and use the above method to rewrite ${\rm ch}(V) \subset \bigcup_i {\rm ch}(V_i)$,
where the $V_i$ are sets of vertices of $P^{(m)}$ that do not contain that
bad conf\/iguration, while not introducing any new bad conf\/igurations.
Iterating this we end up with ${\rm ch}(V) \subset \bigcup_i {\rm ch}(V_i)$ for some
sets $V_i$ without bad conf\/igurations; in particular ${\rm ch}(V)$ is contained
in the right hand side of \eqref{eqpmdecomposition}.

First we consider a bad set of the form $\{w_{ij}, w_{kl}\}$. Then
$p=\frac12 (w_{ij}+w_{kl}) =\frac12 (v_{T_1} + v_{T_2})$, where $T_1$ and
$T_2$ are any two sets of size $|T_i|=m+1$ with $T_1\cup T_2 \cup
\{i,j,k,l\} = \{0,1,\ldots,2m+5\}$. Thus we get $V_1=(V \cup
\{v_{T_1},v_{T_2}\}) \backslash \{w_{ij}\}$ and $V_2 =(V \cup
\{v_{T_1},v_{T_2}\}) \backslash \{w_{kl}\}$, as new sets. In particular the
number of $w$'s decreases and we can iterate this until no bad sets of the
form $\{w_{ij}, w_{kl}\}$ exist.

For the remaining three bad kind of sets we just indicate the way a strictly convex combination of the vectors in the bad set can be written in terms
of better vectors. In each step we assume there are no bad sets of the previous form, to ensure we do not create any new bad sets (at least not of the
form currently under consideration or of a form previously considered).
\begin{enumerate}\itemsep=0pt
\item For $\{w_{ij}, v_S\}$ with $i,j \in S$ we have $\frac23 w_{ij} + \frac13 v_S = \frac13 (v_{T_1}+v_{T_2}+v_{U})$ where
$S\backslash T_1=S\backslash T_2=\{i,j\}$ and $S\cap U = T_1 \cap U=T_2\cap U = \varnothing$ and $T_1 \cap T_2= S\backslash\{i,j\}$ (thus $T_1$, $T_2$ and $U$ cover all the elements of $S$, except
$i$ and $j$, twice, and all other points once).
\item For $\{w_{ij}, w_{ik},v_S\}$ with $i\in S$ we have $\frac13(w_{ij}+w_{ik}+v_S) = \frac13 (v_{T_1}+v_{T_2}+v_U)$
for $S\backslash T_1=S\backslash T_2=\{i\}$ and $S\cap U = T_1 \cap U=T_2\cap U = \varnothing$ and $T_1 \cap T_2= S\backslash\{i\}$ and $j,k \not \in T_1,T_2,U$.
\item For $\{w_{ij}, w_{ik}, v_S\} \subset V$, with $i\in S$, then $j,k\not \in S$ and $\frac13(w_{ij}+w_{ik}+v_S) = \frac13 (v_{T_1}+v_{T_2}+v_U)$
for $S\backslash T_1=S\backslash T_2=\{i\}$ and $S\cap U = T_1 \cap U =T_2\cap U= \varnothing$ and $T_1 \cap T_2= S\backslash\{i\}$ and $j,k \not \in T_1,T_2,U$. \hfill \qed
\end{enumerate}
\renewcommand{\qed}{}
\end{proof}

Let us now consider the bounding inequalities related to these polytopes.
%\marginpar{Note that $\alpha_i\leq 1$ is not bounding for $m=0$. Also $\alpha_r\leq 1+\alpha_0+\alpha_1+\alpha_2$ is not bounding.}
\begin{Proposition}\label{propboundingineq}
The polytopes $P^{(m)}$, $P_{\rm I}^{(m)}$, $P_{\rm II}^{(m)}$ and $P_{\rm III}^{(m)}$ are the subsets
of the hyperplane $\{\alpha\colon \alpha\in \mathbb{R}^{2m+6}|\sum_i \alpha_i = m+1\}$
described by the following bounding inequalities
\begin{itemize}\itemsep=0pt
\item For $P^{(m)}$ the bounding inequalities are
\begin{alignat*}{3}
&-\frac12 \leq \alpha_i \leq 1,\qquad &&(0\leq i\leq 2m+5), & \\
& \alpha_i\le 1+ \alpha_j+\alpha_k+\alpha_l, \qquad && (|\{i,j,k,l\}|=4), & \\
& \alpha_i-\alpha_j \leq 1,\qquad && (i\neq j), & \\
& (|S|-2)\alpha_i+\sum_{j\in S}\alpha_j\ge 0, \qquad &&(i\not\in S,3\le |S|\le m+3).&
\end{alignat*}
For $m=0$ the equations $\alpha_r\leq 1$ and $\alpha_i\le 1+\alpha_j+\alpha_k+\alpha_l$
are valid but not bounding.
\item The polytope $P_{\rm I}^{(m)}$ is described by the bounding inequalities
\[
0\leq \alpha_i \leq 1, \qquad (0\leq i \leq 2m+5).
\]
For this polytope too, if $m=0$ the equations $\alpha_r\leq 1$ are valid but not bounding.
\item
The polytope $P_{\rm II}^{(m)}$ is described by the bounding inequalities
\begin{alignat*}{3}
& -1/2 \leq \alpha_0, &&& \\
& \alpha_r-\alpha_0 \leq 1, \qquad && (r\geq 1), & \\
& (|S|-2)\alpha_0+ \sum_{j\in S}\alpha_j\ge 0, \qquad && (0\not \in S, 0\le |S|\le m+3).&
\end{alignat*}
\item
Finally, the polytope $P_{\rm III}^{(m)}$ is described by the bounding
inequalities
\begin{alignat*}{3}
& \alpha_i+\alpha_j \leq 0, \qquad &&(0\leq i<j\leq 2), & \\
& -\alpha_i \leq \alpha_0+\alpha_1+\alpha_2, \qquad &&(3\leq i \leq 2m+5),& \\
& \alpha_i-1 \leq \alpha_0+\alpha_1+\alpha_2, \qquad &&(3\leq i \leq 2m+5).&
\end{alignat*}
If $m=0$ the equations $\alpha_i-1 \leq \alpha_0+\alpha_1+\alpha_2$ are valid but not bounding.
\end{itemize}
\end{Proposition}

\begin{proof}
It can be immediately verif\/ied that the vertices of the polytopes
$P^{(m)}$, $P_{\rm I}^{(m)}$, $P_{\rm II}^{(m)}$ and~$P_{\rm III}^{(m)}$ satisfy the
relevant inequalities, hence so does any convex linear combination of
them. In particular it is clear that the polytopes are contained in the
sets def\/ined by these bounding inequalities.

Note that the dif\/ferent polytopes have symmetries of $S_{2m+6}$ (for $P^{(m)}$
and $P_{\rm I}^{(m)}$), respectively $S_1 \times S_{2m+5}$ ($P_{\rm II}^{(m)}$),
respectively $S_3 \times S_{2m+3}$ ($P_{\rm III}^{(m)}$). We only have to f\/ind
the bounding inequalities of these polytopes intersected with a
Weyl chamber of the relevant symmetry group, as all bounding inequalities
will be permutations of these. These bounding inequalities can be written
in the form $\mu \cdot \alpha \geq 0$ for each $\alpha$ in the polytope; we
do not need af\/f\/ine equations as we have $\sum_r \alpha_r=m+1$.

A bounding inequality must attain equality at a codimension 1 space of the
vertices of the polytope; in particular if we consider all subsets $V$ of
$2m+5$ vertices of the intersection of each polytope with the relevant Weyl
chamber and insist on $\mu \cdot v=0$ for each $v\in V$, we f\/ind all
bounding inequalities (and perhaps some more inequalities). For $P_{\rm I}$,
$P_{\rm II}$ and $P_{\rm III}$ we are in the circumstance that there are $2m+6$
vertices for the intersection of the Weyl chamber with the polytope; in
particular each set of $2m+5$ vertices corresponds to leaving one vector
out. Moreover the sign of $\mu$ is then determined by insisting on $\mu
\cdot v>0$ for the remaining vertex $v$. As the equations are all
homogeneous the normalization of $\mu$ is irrelevant.

Let us consider the case of $P_{\rm II}$. The set of relevant vertices is
$\{v_S, w_{01},e_1-e_2,\ldots,e_{2m+4}-e_{2m+5}\}$ for $S=\{m+6, \ldots, 2m+5\}$.
We now have the following options for leaving one vector out.
\begin{enumerate}\itemsep=0pt
\item If $\mu \cdot v_S^{(m)}>0$ we get $\mu=\rho + (m+1)e_0$, thus the equation $\alpha_0\geq -1/2$.
\item If $\mu \cdot w_{01}^{(m)} >0$ we get $\mu = -e_0$ and the equation $\alpha_0\leq 0$.
%\item If $\mu \cdot (e_1-e_{2})>0$ we get $\mu = -e_0 + e_1$ and the equation becomes $\alpha_0\leq \alpha_1$;
\item If $\mu \cdot (e_i-e_{i+1})>0$ for $1\leq i\leq m+3$ we get $\mu = (i-2) e_0 + \sum_{r=1}^i e_r$ and the equation becomes
$(i-2)\alpha_0 + \sum_{r=1}^{i} e_r \geq 0$.
\item If $\mu \cdot e_i-e_{i+1}>0$ for $m+4\leq i\leq 2m+4$ we get $\mu = (m+2)(2m+5-i) e_0 + (2m+5-i)\sum_{r=1}^i e_r + (m+4-i)\sum_{r=i+1}^{2m+5} e_r$
and the equation becomes $(\alpha_0 +1)(2m+5-i) \geq \sum_{r=i+1}^{2m+5} \alpha_r$.
\end{enumerate}
Note that the equation $\alpha_0\leq 0$ is the $|S|=0$ case of
$(|S|-2)\alpha_0+\sum_{j\in S}\alpha_j\ge 0$. Now the last set of
equations all follow from the instance $i=2m+4$, i.e.\ $\alpha_0 +1\geq
\alpha_{2m+5}$ and the equation $\alpha_{2m+5} \geq \alpha_{r}$. The rest
are true bounding inequalities. It is only hard to see that the solutions
to $(i-2)\alpha_0 + \sum_{r=1}^{i} e_r = 0$ in the set of vertices of the
polytope span a codimension one space; however the set $\{w_{01}, \ldots,
w_{0i}\} \cup \{ v_{T}~|~ T\subset \{ i+1,\ldots,2m+5\}\}$ does span a set
of codimension one.

In a similar way one obtains the bounding inequalities for $P_{\rm I}$ and
$P_{\rm III}$, we omit the explicit calculations here. To obtain the bounding
inequalities of $P$ itself, we observe that any bounding inequality of $P$
must be a bounding inequality of one of $P_{\rm I}$, $P_{\rm II}$, $P_{\rm III}$ or one
of their permutations, as $P$ is the union of those polytopes. Indeed any
of these equations which are valid on $P$ are bounding inequalities (as the
span of the set of vertices for which equality holds does not reduce in
dimension when going from a smaller polytope to $P$). Thus we can f\/ind the
bounding inequalities for $P$ by checking which of the bounding
inequalities of these smaller polytopes are valid on $P$. This we only need
to check on the vertices of $P$, which is a straightforward calculation.

Note that we could also have obtained the bounding inequalities for $P$ in
the same way that we obtained those of $P_{\rm I}$, $P_{\rm II}$ and
$P_{\rm III}$. However now we would have to take $2m+5$ vectors from the set
$\{v_{S}, w_{01}, e_0-e_1, \ldots, e_{2m+4}-e_{2m+5}\}$, which has $2m+7$
elements. The number of options therefore becomes quite large, thus we
prefer to avoid this method.
\end{proof}

We would like to give special attention to the bounding inequalities of
$P^{(1)}$, which is the polytope which interests us most. We can rewrite
these bounding inequalities in a clearly $W(E_7)$ invariant way.

\begin{Proposition}\label{propboundingineqm1}
The bounding inequalities for $P^{(1)}$ inside the subspace
$\alpha\cdot\rho = 1$ are given by
\begin{alignat*}{3}
& \alpha \cdot \delta \leq 1, \qquad && (\delta \in R(E_7)), &\\
& \alpha \cdot \mu \leq 2,\qquad &&(\mu \in \Lambda(E_8), \mu \cdot \rho=1, \mu \cdot \mu=4)&
\end{alignat*}
for $\alpha \in P^{(1)}$. Here $\Lambda(E_8)= \mathbb{Z}^8 \cup
(\mathbb{Z}^8 + \rho)$ is the root lattice of $E_8$.
\end{Proposition}

\begin{proof}
Up to $S_8$ one can classify the roots of $E_7$, giving $\delta = e_i-e_j$
or $\delta = \rho-e_i-e_j-e_k-e_l$, which handles the bounding inequalities
$\alpha_i-\alpha_j\leq 1$ and $\alpha_i+\alpha_j+\alpha_k+\alpha_l\leq
0$. Similarly we can classify all relevant $\mu\in \Lambda(E_8)$ as
$\mu=2e_i$, $\mu=e_i+e_j+e_k-e_l$, $\mu = \rho - 2e_i$ and $\mu = \rho
+e_i-e_j-e_k-e_l$, the corresponding equations are again directly related
to the bounding inequalities of $P^{(1)}$ as given in Proposition
\ref{propboundingineq}.
\end{proof}

It is convenient to rewrite the integral limits of Propositions
\ref{proptrivlimit} and \ref{propsymbreak} in a uniform way which clearly
indicates the bounding inequalities for the corresponding polytopes
(i.e.\ $P_{\rm I}$, resp. $P_{\rm III}$). It is much harder to give such a uniform
expression for $P_{\rm II}$ (and we need separate expressions for the
intersection with $P_{\rm I}$ and $P_{\rm III}$ and the facet
$\{\alpha_0=1+\alpha_{2m+5}\}$), so we omit those.

\begin{Proposition}\label{propuniformtrivlimit}
Define vectors $v_j = e_j$ and $w_j = e_j - \frac{2}{m+1}\rho $, then the
bounding inequalities for $P_{\rm I}^{(m)}$ become $v_j \cdot \alpha\geq 0$, $w_j
\cdot \alpha \geq 0$ $($and the condition $2\rho \cdot \alpha = m+1)$. The
limit can be written as
\[
B_{\alpha}^m(u) = \prod_{j\neq k \colon v_j \cdot \alpha=v_k\cdot \alpha=0} (u^{v_j+v_k};q) \frac{(q;q)}{2}
\int \frac{(z^{\pm 2};q) \prod\limits_{j\colon w_j \cdot \alpha=0} (u^{w_j} z^{\pm 1};q)}{\prod\limits_{j\colon v_j\cdot \alpha=0} (u^{v_j} z^{\pm 1};q)} \frac{dz}{2\pi i z}.
\]
\end{Proposition}

\begin{Proposition}\label{propuniformsymbreak}
Define the vectors $v_j = e_0+e_1+e_2-e_j$ $(0\leq j\leq 2)$, $w_j =
e_0+e_1+e_2+e_j$ $(3\leq j\leq 2m+5)$, and $x_j = e_0+e_1+e_2-e_j -
\frac{2}{m+1} \rho$ $(3\leq j\leq 2m+5)$, then the bounding inequalities
for $P_{\rm III}^{(m)}$ can be written as $v_j \cdot \alpha\geq 0$, $w_j \cdot
\alpha \geq 0$ and $x_j \cdot \alpha \geq 0$ $($together with $2\rho \cdot
\alpha = m+1)$. Let $y=w_j+x_j$ $($note $y$ is independent of $j)$ then
\begin{gather*}
B_\alpha^m(u) = \frac{\prod\limits_{j\colon w_j \cdot \alpha =0} \prod\limits_{k\colon v_k \cdot \alpha=0} (u^{w_j+u_k};q)}{\prod\limits_{k\colon v_k\cdot \alpha=0} (qu^{v_k};q)}
\left( \prod_{r\neq s\colon v_r\cdot \alpha = v_s\cdot \alpha =0} (qu^{y+v_r+v_s};q) \right)^{1_{\{y\cdot \alpha=0\}}} \\ \phantom{B_\alpha^m(u) =}{} \times
\int \theta(1/z;q) \frac{\prod\limits_{j\colon x_j \cdot \alpha=0} (q^2 z u^{x_j};q)}{\prod\limits_{j\colon w_j\cdot \alpha=0} (zu^{w_j};q)}
\frac{\prod\limits_{r\neq s\colon v_r\cdot \alpha= v_s\cdot \alpha =0} (qu^{v_r+v_s}/z;q)}{\prod\limits_{r\colon v_r\cdot \alpha=0} (u^{v_r}/z;q)} \\
 \phantom{B_\alpha^m(u) =}{}\times
\left( \frac{(1-qu^{y} z^2) \prod\limits_{r\neq s\colon v_r\cdot \alpha=v_s\cdot \alpha=0} (q^2 z u^{y+v_r+v_s};q)}{\prod\limits_{r\colon v_r\cdot \alpha=0} (qzu^{y+v_r};q)} \right)^{1_{\{y\cdot \alpha=0\}}} \frac{dz}{2\pi i z}.
\end{gather*}
\end{Proposition}
\begin{proof}
These two propositions are just a rewriting of Propositions
\ref{proptrivlimit} and \ref{propsymbreak}.
\end{proof}

With these expressions the proof of the following proposition becomes
fairly straightforward.

\begin{Proposition}
Let the polytope $Q$ be either $P_{\rm I}^{(m)}$, $P_{\rm II}^{(m)}$ or
$P_{\rm III}^{(m)}$. For $\alpha\in Q$ the limit in~\eqref{eqlim} exists and
depends only on the face of $Q$ which contains $\alpha$ $($i.e.\ if $\alpha$
and $\beta$ are contained in the same face of $Q$ then $B_{\alpha}^m(u) =
B_{\beta}^m(u))$.
\end{Proposition}

\begin{proof}
Indeed Propositions \ref{propuniformtrivlimit}, respectively
\ref{propuniformsymbreak} give the limits for the vectors $\alpha$ in
$P_{\rm I}^{(m)}$, respectively $P_{\rm III}^{(m)}$. For $P_{\rm II}^{(m)}$ the limits
are given in Proposition \ref{propsum}, except for the cases with
$\alpha_0=0$ (which is the intersection with $P_{\rm I}^{(m)}$), and
$\alpha_1+\alpha_2=0$ (the intersection with $P_{\rm III}^{(m)}$), or a
permutation of such a case. In particular we have obtained limits in those
cases as well.

For $P_{\rm I}$ and $P_{\rm III}$ the expressions in the previous two propositions
immediately show that the limits only depend on which bounding inequalities
are strict or not, and hence on the face of the polytope containing
$\alpha$. For $P_{\rm II}$ we note that the conditions $\alpha_0=0$ and
$\alpha_1+\alpha_2=0$ (governing which proposition to look at) correspond
to bounding inequalities. Within Proposition \ref{propsum} we observe that
the condition $\alpha_r=-\alpha_0$ becomes a bounding equation once
$2\alpha_0= \sum_{r\geq 1\colon \alpha_0+\alpha_r<0} \alpha_r+\alpha_0$ holds
(as the dif\/ference of the equations with $r\in S$ and with $r\not \in S$).
So also for $P_{\rm II}$ the limit only depends on which bounding inequalities
are strict and which not.
\end{proof}

We can now prove the f\/irst of the main theorems, the equivalent result for
the full polytope~$P^{(m)}$.
\begin{proof}[Proof of Theorem \ref{thmlimitsexist}]
By the $S_{2m+6}$ symmetry of $E_m(t)$ we see that if a limit exists for
some~$\alpha$, then it also exists for all permutations of~$\alpha$. As
$P^{(m)}$ is the union of permutations of $P_{\rm I}^{(m)}$, $P_{\rm II}^{(m)}$ and
$P_{\rm III}^{(m)}$, by the previous proposition we f\/ind that the limit
$B_{\alpha}^m$ exists for all $\alpha \in P^{(m)}$. We would like to extend
the statement about dependence on faces as well. To prove this it would be
suf\/f\/icient to show that all faces of $P^{(m)}$ are in fact a face of one of
the polytopes in its decomposition, however this is not true. We do have
the following lemma.

\begin{Lemma}
All faces of $P^{(m)}$ are faces of either $P_{\rm I}^{(m)}$,
$\sigma(P_{\rm II}^{(m)})$ or $\sigma(P_{\rm III}^{(m)})$ for some $\sigma \in
S_{2m+6}$, except the interior of $P^{(m)}$ and permutations of the facet
given by the equality $\alpha_0+\alpha_1+\alpha_2+\alpha_3=0$.
\end{Lemma}

\begin{proof}
Any face of $P^{(m)}$ can be written as the set of all convex linear
combinations of some set $V$ of vertices of $P^{(m)}$. Recall that $V$ is
contained in the set of vertices of $P_{\rm I}$, $P_{\rm II}$ or $P_{\rm III}$ (or a~permutation thereof), unless it contains one of the four bad sets in the
proof of Proposition~\ref{propdecomp}. Therefore, except when $V$ contains
bad sets, the face determined by $V$ is a face of $P_{\rm I}$, $P_{\rm II}$ or~$P_{\rm III}$. We now show that if $V$ contains a bad set, the convex hull of
$V$ contains a point in the interior of $P^{(m)}$ or the interior of the
facet given by $\alpha_0+\alpha_1+\alpha_2+\alpha_3=0$. Hence ${\rm ch}(V)$ is
either equal to the interior of $P^{(m)}$, or to the special facet.
\begin{enumerate}\itemsep=0pt
\item If $\{w_{ij}, v_{S_1},v_{S_2}\} \subset V$ ($i \in S_1$, $j\in S_2$)
 then $(w_{ij}+v_{S_1}+v_{S_2})/3 \in {\rm ch}(V)$. This is a point where all
 elements are $1/6$, $1/2$ or $5/6$, in particular it is a point in the
 interior of $P_{\rm I}$, and thus of $P$ itself.
\item If $\{w_{ij}, v_S\}\subset V$ ($i,j \in S$), then $(w_{ij}+v_S)/2\in
 {\rm ch}(V)$, which is in the interior of $P_{\rm I}$.
\item If $\{w_{ij}, w_{ik}, v_S\}\subset V$, ($i\in S$), then
 $(w_{ij}+w_{ik}+2v_S)/4 \in {\rm ch}(V)$, which is again a point in the
 interior of $P_{\rm I}$.
\item If $\{w_{ij}, w_{kl}\}\subset V$, then $(w_{ij}+w_{kl})/2\in
 {\rm ch}(V)$. All bounding inequalities of $P$ are strict on this point except
 $\alpha_i+\alpha_j+\alpha_k +\alpha_l=0$, thus it is a point in the
 interior of the corresponding facet. Thus ${\rm ch}(V)$ is either this facet or
 the interior of $P$.\hfill \qed
\end{enumerate}
\renewcommand{\qed}{}
\end{proof}
Now it remains to show that the function $B_{\alpha}^m$ is the same for all
points in the interior, and all points on the facet given by
$\alpha_0+\alpha_1+\alpha_2+\alpha_3=0$. This follows from the following
lemma.

\begin{Lemma}
On the facet of $P^{(m)}$ given by $\alpha_0+\alpha_1+\alpha_2+\alpha_3=0$ we
have $B_{\alpha}^m(u) = (u_0u_1u_2u_3;q)$. Moreover on the interior of
$P^{(m)}$ we have $B_{\alpha}^m(u)=1$.
\end{Lemma}

\begin{proof}
In the $m=0$ case we f\/ind that the right hand side of the evaluation
formula \eqref{eqeval} converges for $\alpha$ on the facet to $(q/u_4u_5;q)
= (u_0u_1u_2u_3;q)$, while in the interior of $P^{(m)}$ the limit converges to
1 (as for $m=0$ the condition $\alpha_r+\alpha_s=1$ is equivalent to the
sum of the other four parameters being zero.). Thus for $m=0$ the lemma is
true.

For $m>0$ we can classify all the faces of $P_{\rm I}^{(m)}$, $P_{\rm II}^{(m)}$ and
$P_{\rm III}^{(m)}$ which intersect the given facet and the interior of $P^{(m)}$.
The bounding inequality $\alpha_0+\alpha_1+\alpha_2+\alpha_3\geq 0$ implies
that the only vertices allowed in the closure of this facets are $v_S$ for
$0,1,2,3 \not \in S$ and $w_{ij}$ for $i,j \in \{0,1,2,3\}$. We obtain the
following set of faces of $P_{\rm I}$, $P_{\rm II}$ and $P_{\rm III}$ in the facet
$\alpha_0+\alpha_1+\alpha_2+\alpha_3=0$ modulo permutations of the
parameters.

 \begin{center}
 \begin{tabular}{l|l|l}
 Polytope & Vertices & Relations \\
 \hline
 \tsep{0.5ex} $P_{\rm I} \cap P_{\rm II} \cap P_{\rm III}$ & $v_S$ ($0,1,2,3\not \in S$) & $\alpha_0=\alpha_1=\alpha_2=\alpha_3=0$ \\
 $P_{\rm II} \cap P_{\rm III}$ & $w_{01}$, $v_S$ ($0,1,2,3 \not \in S$) &
 $\alpha_0=\alpha_1=-\alpha_2=-\alpha_3$ \\
 & $w_{01}$, $w_{02}$, $v_S$ ($0,1,2,3 \not \in S$) & $ \alpha_0+\alpha_3=0$, $\alpha_1=\alpha_2=0$ \\
 $P_{\rm II}$ & $w_{01}$, $w_{02}$, $w_{03}$, $v_S$ ($0,1,2,3\not \in S$) & $\alpha_0<\alpha_r < -\alpha_0$ ($r=1,2,3$) \\
 $P_{\rm III}$ & $w_{01}$, $w_{02}$, $w_{12}$, $v_S$ ($0,1,2,3 \not \in S$) &
 $-\alpha_3< \alpha_r < \alpha_3$ ($r=0,1,2$)
 \end{tabular}
 \end{center}

We omitted the conditions on $\alpha_r$ for $r\geq 4$ as they are the same
in each case. Indeed the bounding inequalities imply that $-\beta <
\alpha_r<1+\beta$ for $r\geq 4$, where $\beta$ is the sum of the three
smallest parameters. The classif\/ication becomes apparent once we realize
that the $v_S$ part of the vertices of the faces is f\/ixed (they must
contain a point $v_S$ with $0,1,2,3\not \in S$ to be in the open facet,
while they cannot contain any other $v_S$). Thus we only have to consider
the possibilities for adding some $w_{ij}$'s. In these f\/ive faces we can
directly check what the limit is, and observe that the corresponding
integrals and sums indeed evaluate to the desired $(u_0u_1u_2u_3;q)$. The
required evaluation identities are provided by the $m=0$ cases of these
faces.

Similarly we can describe all faces of $P_{\rm I}$, $P_{\rm II}$ and $P_{\rm III}$
(modulo permutations) meeting the interior of $P$.
\begin{center}
 \begin{tabular}{l|l|l}
 Polytope & Vertices & Relations \\
 \hline
 \tsep{0.5ex} $P_{\rm I} \cap P_{\rm II} \cap P_{\rm III}$
 & $v_S$, ($0,1,2\not \in S$)
 & $\alpha_0=\alpha_1=\alpha_2=0$ \\
 $P_{\rm I} \cap P_{\rm II}$
 & $v_S$ ($0,1\not \in S$)
 & $\alpha_0=\alpha_1=0<\alpha_2<1$ \\
 & $v_S$ ($0\not \in S$) & $\alpha_0=0<\alpha_1,\alpha_2<1$\\
 $P_{\rm I}$ & $v_S$ & $0<\alpha_0,\alpha_1,\alpha_2<1$ \\
 $P_{\rm II} \cap P_{\rm III}$ & $w_{01}$, $v_S$ ($0,1,2\not \in S$) & $-1/2 < \alpha_0=\alpha_1=-\alpha_2$ \\
 & $w_{01}$, $w_{02}$, $v_S$ ($0,1,2\not \in S$) & $-1/2<\alpha_0< \alpha_1=-\alpha_2$
 \\
 $P_{\rm III}$ & $w_{01}$, $w_{02}$, $w_{12}$, $v_S$ ($0,1,2\not \in S$) &
 $\alpha_r+\alpha_s<0$, ($r,s\in \{0,1,2\}$) \\
 $P_{\rm II}$ & $w_{01}$, $v_S$ ($0,1,\not \in S$) & $-1/2<\alpha_0=\alpha_1>-\alpha_2$ \\
 & $w_{01}$, $w_{02}$, $w_{03}$, $v_S$ ($0 \not \in S$) & $\alpha_0<0$ and $\alpha_0<\alpha_1> -\alpha_2$
 \end{tabular}
\end{center}

For each face we have the extra conditions $-\beta< \alpha_r < 1+\beta$ for
$r\geq 3$.
We can check for all these 9 faces that the value on that face equals 1
identically. Again the required evaluations all follow from the $m=0$ case.
\end{proof}

We conclude that also on the two types of faces of $P^{(m)}$ which are not
a face of one of the subpolytopes, the value is the same on the entire
face.
\end{proof}

We can now also prove the second main theorem, the iterated limit property.

\begin{proof}[Proof of Theorem \ref{thmiteratedlimits}]
It is suf\/f\/icient to prove this property for the closed polytopes
$P_{\rm I}^{(m)}$, $P_{\rm II}^{(m)}$ and $P_{\rm III}^{(m)}$. Observe that for the
vector $t\alpha+ (1-t)\beta$ precisely those boundary conditions (of the
respective polytopes) are strict which are strict for either $\alpha$ or
$\beta$.

In Propositions \ref{propuniformtrivlimit} and \ref{propuniformsymbreak} we
have written the limits in $P_{\rm I}$ and $P_{\rm III}$ as an integral of a
product of terms $f(u,w) = (u^{w};q)^{1_{\{w\cdot \alpha =0\}}}$ where $w
\cdot \alpha\geq 0$ is the sum of some bounding inequalities. In particular
in the expression for $B_{t\alpha+(1-t)\beta}^m$ only those terms remain
corresponding to sums of bounding inequalities which are attained in both
$\alpha$ and $\beta$. On the other hand, as $(x^{\beta-\alpha}u)^w =
x^{(\beta-\alpha)\cdot w} u^w$ we f\/ind that if $w\cdot \alpha=w\cdot
\beta=0$ then the term $f_{\alpha}(x^{\beta-\alpha}u,w)$ is constant in $x$
and thus does not change in the limit, while if $w\cdot \alpha>0$, we f\/ind
that $f_{\alpha}(x^{\beta-\alpha}u,w)=1$ is again independent of $x$, and
f\/inally if $w\cdot \alpha =0 < w\cdot \beta$, then we have a uniform limit
$\lim_{x\to 0} f_{\alpha}(x^{\beta-\alpha}u,w) = \lim_{x\to 0}
(x^{\beta\cdot w} u^w;q) =1 = f_{t\alpha + (1-t)\beta}(u,w)$. Thus to prove
iterated limits we only have to show we are allowed to interchange limit
and integral, which follows from the fact that we integrate over some
compact contour, and the fact that the $x$-dependent poles which have to
remain inside the contour converge to 0, while the $x$-dependent poles
which have to remain outside the contour go to inf\/inity; in particular for
$x$ small enough we can take an $x$-independent contour.

To prove the result for $P_{\rm II}$ is more complicated as we do not have a
uniform description of the limit. For the closed facets determined by
$\alpha_0=0$, $\alpha_1+\alpha_2=0$ or $\alpha_0=\alpha_1$ we have an
integral description (see Propositions \ref{proptrivlimit},
\ref{propsymbreak} and \ref{propsumintegral}), and limits within these
facets can be treated as for $P_{\rm I}$ and $P_{\rm III}$. For the complement of
these facets the limit is given in Proposition~\ref{propsum} as a~single
sum (possibly of only one term), and we can replicate the argument for
$P_{\rm I}$ and $P_{\rm III}$ for sums instead of integrals to see that iterated
limits within the complement of the facets hold. We are left with showing
that limits from one of the three facets to the inside behave correctly.

These limits all pass from an integral to a single sum. The simplest
argument we found is to simply calculate these limits in all three cases
separately (due to the rather uniform expressions of the integrals and
single sums, we can handle the dif\/ferent faces in each of the closed facets
uniformly). It boils down to picking the residues associated with poles
which either go to zero as $x\to 0$, while they should be outside the
contour, or go to inf\/inity while they should be inside the
contour. Subsequently we bound the integrand around the unit circle to show
that the remaining integral vanishes in the limit. Finally we give a bound
on the residues and use dominated convergence to show we are allowed to
interchange limit and sum. This bound also serves to show that the sum of
the residues which are not associated to $u_0$ (where we take the f\/irst coef\/f\/icient
in $\alpha$ to be the strictly lowest one) vanishes as well. The
calculations involved are again tedious, and very similar to the
calculations in the proof of Proposition~\ref{propsum}.
\end{proof}

\section[Transformations: the $m=1$ case]{Transformations: the $\boldsymbol{m=1}$ case}\label{sectrafo}

In this section we start harvesting the results we can now immediately
obtain given this picture of basic hypergeometric functions as faces of the
polytope $P^{(1)}$. This was already done for the top two levels (i.e.\ the
functions corresponding to vertices and edges) by Stokman and the authors
in~\cite{vdBRS}, though there we did not yet see the polytope. In the next
section we give a worked through example (related to ${}_2\phi_1$) of the
abstract results in this section. As a convenient tool to understand the
implications of the results mentioned, we refer to Appendix \ref{secapp1}
which contains a~list of all the possible functions $B_{\alpha}^1$.

Recall that we have a basic hypergeometric function attached to each face
of the polytope~$P^{(1)}$, and that the Weyl group $W(E_7)$ acts both on
$P^{(1)}$ and on sets of parameters $\tilde{\mathcal{H}}_1$. As the
elliptic hypergeometric function is invariant under this action we
immediately obtain

\begin{Theorem}
Let $w\in W(E_7)$, $\alpha \in P_1$ and $u\in \tilde{\mathcal{H}}_1$ then
\[
B_{\alpha}^1(u) = B_{w(\alpha)}^1(w(u)).
\]
\end{Theorem}
\begin{proof}
Indeed we have
\begin{gather*}
B_{\alpha}^1(u) = \lim_{p\to 0} E^1(p^{\alpha} \cdot u) = \lim_{p\to 0} E^1(w(p^{\alpha} \cdot u))
= \lim_{p\to 0} E^1(p^{w(\alpha)}\cdot w(u)) = B_{w(\alpha)}^1(w(u)). \tag*{\qed}
\end{gather*}
\renewcommand{\qed}{}
\end{proof}
This gives us formulas of two dif\/ferent kinds for the functions $B_{\alpha}^1$.

First of all we can obtain the symmetries of a function by considering the
stabilizer of the corresponding face with respect to $W(E_7)$. This
includes for example Heine's transformation of a ${}_2\phi_1$,
transformations of non-terminating very-well-poised ${}_8\phi_7$'s, and
Baileys' four-term relation for very-well-poised ${}_{10}\phi_9$'s (as a
symmetry of a sum of two ${}_{10}\phi_9$'s).
 %\cite[(III.1-3,9-10,23-24,39)]{GandR}

The symmetry group of the related function is the stabilizer of a generic
point in the face, or equivalently the stabilizer of all the vertices of
the face. Indeed if some element of $w$ f\/ixes the face, but non-trivially
permutes the vertices of the face, then it can be written as the product of
a permutation of the vertices generated by ref\/lections in hyperplanes
orthogonal to the face, and a Weyl group element which stabilizes the
vertices of the face. However, as the functions~$B_{\alpha}^1(u)$ only
depend on the space orthogonal to the face, the f\/irst factor has no ef\/fect.

Secondly, elements of the Weyl group which send one face to a dif\/ferent
face induce transformations relating the two functions associated to these
two faces. Examples of this include Nassrallah and Rahman's integral
representation of a very-well-poised ${}_8W_7$, the expression of this
function as a sum of two balanced ${}_4\phi_3$'s, and a relation relating
the sum of two ${}_3\phi_2$'s with argument $z=q$ to a single ${}_3\phi_2$
with $z=de/abc$ \cite[(III.34)]{GandR}.

Indeed all simplicial faces of $P^{(1)}$ of the same dimension are related
by the $W(E_7)$ symmetry, except for dimension~5. Indeed for dimension 5
there are two orbits: $5$-simplices which bound a~6-simplex and those which
do not. In particular in Fig.~\ref{degscheme1} below, there exist
transformation formulas between all functions on the same horizontal level,
except for the second-lowest level, where you must distinguish between
those faces which are at the boundary of some higher dimensional simplicial
face, and those that are not. As two functions between which there exists a
transformation formula have the same symmetry group, we have written down
the symmetry groups of all the functions on each level on the left hand
side. Note that the symmetry group for 5-simplices at the boundary of a
6-simplex is $1$ (i.e.\ the group with only 1 element), while for the other
5-simplices the symmetry group is $W(A_1) \cong S_2$.

We can also consider the limit of the contiguous relations satisf\/ied by
$E^1$. The $q$-contiguous relations reduce to $q$-contiguous relations. We
get a relation for each set of three terms $B_{\alpha}(u \cdot
q^{\beta_i})$, where the $\beta_i$ are projections of points in the root
lattice of $E_7$ to the space orthogonal to the face containing
$\alpha$. %\marginpar{Talk
%more about this?}
%This projection is again a lattice, in fact a lattice associated to a root system,
%and thus we f\/ind the correct lattice for obtaining $q$-contiguous relations.

More interesting is the limit of a $p$-contiguous relation. In order for us
to be able to take a limit we have to f\/ind three points on $P_1$ whose
pairwise dif\/ferences are roots of $E_7$.

\begin{Proposition}
Let $\alpha, \beta, \gamma\in P^{(1)}$ be such that $\alpha-\beta,
\alpha-\gamma, \beta-\gamma\in R(E_7)$ and form an equilateral triangle
$($i.e.\ $(\alpha-\beta) \cdot (\alpha-\gamma) = 1)$, and let $u\in
\tilde{\mathcal{H}}_1$. Recall $S=\{ v\in R(E_8) ~|~ v\cdot \rho =1\}$.
Then
\begin{gather}
\prod_{\substack{\delta \in S \\ \delta \cdot (\alpha,\beta,\gamma) = (1,0,0)}} (u^{\delta};q)
u^{\gamma} \theta(u^{\beta-\gamma};q) B^1_\alpha(u)
+ \prod_{\substack{\delta \in S \\ \delta \cdot (\alpha,\beta,\gamma) = (0,1,0)}} (u^{\delta};q)
u^{\alpha} \theta(u^{\gamma-\alpha};q) B^1_\beta(u) \nonumber\\
 \qquad \qquad
{}+ \prod_{\substack{\delta \in S\\ \delta \cdot (\alpha,\beta,\gamma) = (0,0,1)}} (u^{\delta};q)
u^{\beta} \theta(u^{\alpha-\beta};q) B^1_\gamma(u) =0.\label{eq3term}
%\notag
\end{gather}
\end{Proposition}

Note that \eqref{eq3term} is written in its most symmetric form. In order
to avoid non-integer powers of the constants one should f\/irst multiply the
entire equation by $u^{-\alpha}$ and use $u^{\rho}=q$.

\begin{proof}
Choose $\zeta$ such that $\tilde \alpha = \alpha+\zeta$, $\tilde \beta=\beta+\zeta$ and
$\tilde \gamma = \gamma+\zeta$ are all roots of $E_7$. This is possible by choosing
$\tilde \alpha$ such that $\tilde \alpha \cdot (\alpha-\beta)=1$ and $\tilde \alpha \cdot (\alpha-\gamma)=1$,
and we can always f\/ind roots satisfying these two conditions. Now observe that
$\tilde \alpha \cdot \tilde \beta \leq 1$ as inner product of two dif\/ferent roots of $E_7$, and that
\[
\tilde \alpha \cdot \tilde \beta = (\alpha + \zeta)\cdot (\beta+\zeta) =
(\alpha+\zeta)\cdot (\alpha+\zeta) + (\alpha+\zeta) \cdot (\beta-\alpha)
\geq 2 -1 =1,
\]
as $\alpha + \zeta \neq - (\beta-\alpha)$ (equality here would imply $\beta+\zeta=0$).
Thus $\tilde \alpha \cdot \tilde \beta=1$, and hence
$\tilde \alpha$, $\tilde \beta$ and $\tilde \gamma$ satisfy the conditions of Theorem \ref{thcontell}.
Setting $t=u\cdot p^{\zeta}$ in \eqref{eqpcont} we obtain
\begin{gather}
\prod_{\substack{\delta\in S \\ \delta \cdot (\alpha-\beta)=\delta \cdot (\alpha-\gamma) = 1}} (u^{\delta} p^{\delta\cdot \beta};q)
u^{\gamma} p^{\gamma\cdot \zeta} \theta(u^{\beta-\gamma}p^{\zeta\cdot(\beta-\gamma)} ;q) E^1(u\cdot p^{\alpha})
\nonumber\\
\qquad\qquad{}+\prod_{\substack{\delta\in S \\ \delta \cdot (\beta-\alpha)=\delta \cdot (\beta-\gamma) = 1}} (u^{\delta} p^{\delta \cdot \gamma};q)
u^{\alpha} p^{\alpha\cdot \zeta} \theta(u^{\gamma-\alpha} p^{\zeta \cdot (\gamma-\alpha)};q) E^1(u\cdot p^{\beta})
\nonumber\\
\qquad\qquad{}+\prod_{\substack{\delta\in S \\ \delta \cdot (\gamma-\alpha) = \delta \cdot (\gamma-\beta) = 1}} (u^{\delta}p^{\delta\cdot\alpha};q)
u^{\beta} p^{\beta \cdot \zeta}
\theta(u^{\alpha-\beta} p^{\zeta \cdot (\alpha-\beta)};q) E^1(u\cdot p^{\gamma})=0 .\label{eqconttemp}
\end{gather}
for $u\in \tilde{\mathcal{H}}_1$. Now we prove a lemma

\begin{Lemma}
Let $\alpha$, $\beta$ and $\gamma$ be as in the Proposition and let $\delta
\in S$ satisfy $\delta \cdot (\beta-\gamma)=0$ then $\delta \cdot \beta
\geq 0$. Moreover $\zeta \cdot (\beta-\gamma)=0$ for $\zeta$ as in this proof.
\end{Lemma}

\begin{proof}
Recall the bounding inequalities for $P^{(1)}$ given in Proposition
\ref{propboundingineqm1}. Note that $\mu=\rho-\gamma+\beta-\delta \in
\Lambda(E_8)$ satisf\/ies $\mu \cdot \rho=1$ and $\mu\cdot \mu=4$, thus we
get $\beta \cdot (\rho-\gamma+\beta-\delta) \leq 2$ and similarly
$\gamma\cdot (\rho+\gamma-\beta-\delta) \leq 2$. Adding these two
inequalities and simplifying gives $\delta \cdot (\beta + \gamma) \geq 0$,
and as $(\beta -\gamma) \cdot \delta=0$, this implies $\delta \cdot \beta
\geq 0$.

Now observe that
\[
\zeta\cdot (\beta-\gamma) = (\tilde \alpha - \alpha) \cdot (\beta-\gamma) =
\tilde \alpha \cdot (\tilde \beta-\tilde \gamma) - \alpha \cdot (\beta-\gamma)
=-\alpha \cdot (\beta-\gamma),
\]
where in the last equality we used that $\tilde \alpha \cdot \tilde \beta=1=\tilde \alpha \cdot \tilde \gamma$. Thus we need to show that $\alpha \cdot (\beta-\gamma)=0$.
By the bounding inequalities we have $\alpha \cdot (\alpha-\beta) \leq 1$,
but also
\[
\alpha \cdot (\alpha-\beta) = (\alpha-\beta)\cdot (\alpha -\beta) - \beta\cdot(\beta-\alpha) = 2 - \beta\cdot(\beta-\alpha) \geq 1.
\]
Thus we f\/ind $\alpha \cdot (\alpha -\beta)=1$. By symmetry we also have
$\alpha \cdot (\alpha -\gamma)=1$. Thus it follows that $\alpha \cdot
\beta = \alpha \cdot \gamma$, or $\alpha \cdot (\beta-\gamma)=0$.
\end{proof}

The lemma shows that $p^{\gamma \cdot \zeta}=p^{\alpha \cdot \zeta} = p^{\beta \cdot \zeta}$, so we can
divide by this term. Using this lemma we see that we can subsequently take the limit $p\to 0$ in
\eqref{eqconttemp} directly as the arguments of the $\theta$ functions do
not depend on $p$, while the arguments of the $q$-shifted factorials are
either independent of $p$ or vanish as $p\to 0$.
\end{proof}

The relations obtained in this way are three-term relations. By the
geometry of the polytope the $\alpha$, $\beta$ and $\gamma$ in the above
proposition must be such that the faces they are contained in, are in the
same $W(E_7)$ orbit. In particular we can rewrite our three term relation
as a relation between three instances of the same function. Thus we get as
examples three-term relations for ${}_3\phi_2$'s \cite[(III.33)]{GandR}.
Moreover we obtain the six-term relations of ${}_{10}\phi_9$'s as studied
in \cite{GM} and \cite{LvdJ}.

One reason why the $p$-contiguous relations morally should exist on the
elliptic level is that the three functions related by $p$-shifts in roots
of $E_7$ satisfy the same second order $q$-dif\/ference equations (after a
suitable gauge transformation). In particular we can take the limit of
these $q$-dif\/ference equations and see that $B_{\alpha}^1$, $B_{\beta}^1$
and $B_{\gamma}^1$ also satisfy the same second order $q$-dif\/ference
equations. In a very degenerate case, there exist faces for which to a
vector $\alpha$ in that face there exists exactly one root $r \in R(E_7)$
such $\alpha+r \in P^{(1)}$. In particular, while we cannot f\/ind a three
term relation in this case, we do obtain the second solution of the
corresponding $q$-dif\/ference equations. In the general case we can obtain
the symmetry group of the $q$-dif\/ference equations by looking at the
stabilizer of the shifted lattice $\alpha + \Lambda(E_7)$ for a generic
point $\alpha$ in the face. This stabilizer, the stabilizer of $\alpha$
under the {\em affine} Weyl group, is denoted the af\/f\/ine symmetry group in
Fig.~\ref{degscheme1}.

\section[An extended example: ${}_2\phi_1$]{An extended example: $\boldsymbol{{}_2\phi_1}$}\label{Heine}

In this section we consider the simplicial face with vertices $w_{01}$,
$w_{02}$, $v_{67}$ and $v_{57}$. The centroid of this face is the point
$\alpha=(-1/4,0,0,1/4,1/4,1/2,1/2,3/4)$, and we f\/ind using Proposition~\ref{propsymbreak} %and \ref{propsymbreaksum}
that the limit can be
expressed as
\begin{gather*}
B_{\alpha}^1(u) = \frac{(q,u_3u_0,u_4u_0;q)}{(q/u_1u_2;q)} \int \theta(u_0u_1u_2/z;q)
\frac{(qz/u_7;q)}{(u_0/z,u_3z,u_4z;q)} \frac{dz}{2\pi i z} \\
\phantom{B_{\alpha}^1(u)}{} =
(u_1u_2,qu_0/u_7;q) {}_2\phi_1\left( \begin{array}{c} u_0u_3,u_0u_4 \\ qu_0/u_7\end{array} ;q,u_1u_2\right),
\end{gather*}
as long as this series converges (this integral expression for a
${}_2\phi_1$ is not very exciting, as it is related to the series by
picking up the residues upon moving the integration contour to zero).

The stabilizer group of this face under the $W(E_7)$ action equals the
stabilizer of $\alpha$ (as it should be a permutation of the four vertices
of the face). However those ref\/lections in $W(E_7)$ which non-trivially
permute the vertices of the face are in roots which are the dif\/ference of
two vertices, so they will just induce a shift along a vector in the face;
as our functions only depend on the space orthogonal to the face, they act
as identity on our function (for example they permute $u_1\leftrightarrow
u_2$ or $u_5\leftrightarrow u_6$). Thus we are only interested in those
elements of $W(E_7)$ which leave the four vertices of this face
invariant. In particular, this includes (and by Coxeter theory, is
generated by) the ref\/lections in the hyperplanes orthogonal to the roots
$\{ \pm(e_3-e_4),\pm (\rho-e_0-e_3-e_4-e_7), \pm (\rho-e_0-e_3-e_5-e_6),
\pm (\rho-e_0-e_4-e_5-e_6)\}$. These eight roots form the root system of
$A_2\times A_1$, thus the symmetry group of a ${}_2\phi_1$ is $W(A_2\times
A_1)$, or the permutation group $S_3\times S_2$. For example the ref\/lection
$s_{\rho-e_0-e_3-e_4-e_7}$ generates the symmetry $u\mapsto
(u_0/s,u_1s,u_2s,u_3/s,u_4/s,u_5s,u_6s,u_7/s)$, with $s=\sqrt{u_0
 u_3u_4u_7/q}$, or
\begin{gather*}
(u_1u_2,qu_0/u_7;q) {}_2\phi_1\left( \begin{array}{c} u_0u_3,u_0u_4 \\ qu_0/u_7\end{array} ;q,u_1u_2\right)
\\ \qquad{} =
(u_0u_1u_2u_3u_4u_7/q,qu_0/u_7;q) {}_2\phi_1\left( \begin{array}{c} q/u_4u_7,q/u_3u_7 \\ qu_0/u_7\end{array} ;q,u_0u_1u_2u_3u_4u_7/q \right) ,
\end{gather*}
or simplifying we get
\begin{equation*}%\label{eqtrafo2phi1}
(z,c;q) {}_2\phi_1\left( \begin{array}{c} a,b \\ c\end{array} ;q,z\right)
%\\
=
(abz/c, c;q) {}_2\phi_1\left( \begin{array}{c} c/b,c/a \\ c\end{array} ;q,abz/c\right) ,
\end{equation*}
which is one of Heine's transformations \cite[(III.3)]{GandR}. Similarly
related to $s_{\rho-e_0-e_4-e_5-e_6}$ we obtain
\[
(z,c;q) {}_2\phi_1\left( \begin{array}{c} a,b \\ c\end{array} ;q,z\right)
%\\
 =
(c/a, az;q) {}_2\phi_1\left( \begin{array}{c} a,abz/c \\ az\end{array} ;q,c/a\right) ,
\]
another one of Heine's transformations \cite[(III.2)]{GandR}. Together with
the permutation swapping $a$ and $b$ (given by $s_{e_3-e_4}$) these two
transformations generate the entire symmetry group.

As for transformations to other functions, there are no less than 6 other
faces in the $W(E_7)$-orbit of the face containing $\alpha$ up to $S_8$
symmetry. The related transformations are given by (after some
simplif\/ication)
%We f\/ind (already simplif\/ied),
\begin{gather*}
(z,c;q) {}_2\phi_1\left( \begin{array}{c} a,b \\ c\end{array} ;q,z\right)
%%%%%%%%%%%%%%%%%%%%%%%%%%%%%%%%%%%%%%%%
 %\\&
 = (bz,c;q) {}_2\phi_2 \left( \begin{array}{c} b,c/a \\ bz,c \end{array} ;q,az\right)
 % Reflection in x_0357
%%%%%%%%%%%%%%%%%%%%%%%%%%%%%%%%%%%%%%%%%
\\ \qquad{} = \frac{(a,b,abz/c,q;q)}{2} \int \frac{(y^{\pm 2}, \sqrt{cz} y^{\pm 1};q)}{(\sqrt{c/z} y^{\pm 1}, a\sqrt{z/c} y^{\pm 1}, b\sqrt{z/c} y^{\pm 1};q)} \frac{dy}{2\pi i y}
% Reflection in x_0127 = \rho - e_0-e_1-e_2-e_7
%%%%%%%%%%%%%%%%%%%%%%%%%%%%%%%%%%%%%%%%%
\\ \qquad{} = \frac{(z,c/b,c/a;q)}{(c/ab;q)}
\rphis{3}{2}
%{}_3\phi_2
\left( \begin{array}{c} abz/c, a,b \\ qab/c,0 \end{array};q,q\right) +
\frac{(a,b,abz/c;q)}{(ab/c;q)} {}_3\phi_2\left( \begin{array}{c} z,c/a,c/b \\ qc/ab,0 \end{array};q,q\right)
% Reflection in x_0256
%%%%%%%%%%%%%%%%%%%%%%%%%%%%%%%%%%%%%%%%%
\\ \qquad{} = \frac{(z,abz/c,c;q)}{(bz/c;q)} {}_3\phi_2\left( \begin{array}{c} c/b,a,0 \\ qc/bz,c \end{array};q,q\right)
+ \frac{(c/b,a,bz;q)}{(c/bz;q)} {}_3\phi_2\left( \begin{array}{c} z, abz/c,0 \\ qbz/c,bz \end{array} ;q,q\right)
 % Reflection in x_0247
%%%%%%%%%%%%%%%%%%%%%%%%%%%%%%%%%%%%%%%%%
 \\ \qquad{} =
%%%%%% WRONG! %%%%%%%%%%%% \frac{ (bz/q,az/q,c;q)}{(qw;q)}
%%%%%%%%%%%%%%%%%%%%%%%%%%\rWsn{6}{5}{2}( w;a,b,abz/c;q,w qc/ab)
\frac{(az,bz,c;q)}{(abz;q)} \rWsn{6}{5}{2}(\frac{abz}{q}; a,b,abz/c;q,cz)
 % Reflection in x_0567
%%%%%%%%%%%%%%%%%%%%%%%%%%%%%%%%%%%%%%%%%
 \\ \qquad{}= (q;q) \int \theta(z/y;q) \frac{(cy,aby;q)}{(ay,by, cy/z;q)} \frac{dy}{2\pi iy}.
 %Reflection in x_0346
\end{gather*}

Let us now consider the three term relations (as limit of $p$-contiguous relations). The points~$\beta$ in the
polytope with $\alpha-\beta \in \Lambda(E_7)$ in the root lattice of $E_7$ are
\begin{center}
\begin{tabular}{l|l}
$\beta$ & $B_{\beta}^1$ \\
\hline
\tsep{2ex} $(-\frac{1}{4},0,0,\frac{1}{4},\frac{1}{4},\frac{1}{2},\frac{1}{2},\frac{3}{4})=\alpha$ & $
(u_1u_2,qu_0/u_7;q) {}_2\phi_1\left( \begin{array}{c} u_0u_3,u_0u_4 \\ qu_0/u_7\end{array} ;q,u_1u_2\right)
$ \\
\tsep{2ex} $(\frac34,0,0,\frac14,\frac14,\frac12,\frac12,-\frac14)$ &
$
(u_1u_2,qu_7/u_0;q) {}_2\phi_1\left( \begin{array}{c} u_7u_3,u_7u_4 \\ qu_7/u_0\end{array} ;q,u_1u_2\right)
 $
 \\
\tsep{2ex} $(\frac14,\frac12,\frac12,\frac34,-\frac14,0,0,\frac14)$ &
$
(u_5u_6,qu_4/u_3;q) {}_2\phi_1\left( \begin{array}{c} u_0u_4,u_7u_4 \\ qu_4/u_3\end{array} ;q,u_5u_6\right)
 $
\\
\tsep{2ex} $(\frac14,\frac12,\frac12,-\frac14,\frac34,0,0,\frac14)$ &
$
(u_5u_6,qu_3/u_4;q) {}_2\phi_1\left( \begin{array}{c} u_0u_3,u_7u_3 \\ qu_3/u_4\end{array} ;q,u_5u_6\right)
 $
\end{tabular}
\end{center}
Here we can use the balancing condition $\prod_r u_r =q^2$ to rewrite $u_5u_6$ in terms of the previous parameters.
Any three of these four functions now give a three term relation, for example (after simplif\/ication)
\begin{gather*}
(bq/c,q/a,c;q) \frac{az}{c} \theta(c/az;q) {}_2\phi_1 \left( \begin{array}{c}
a,b \\ c \end{array} ;q, z\right) \\
 \qquad{} + (b,c/a,q^2/c;q) \frac{q}{c} \theta(az/q;q) {}_2\phi_1 \left(
 \begin{array}{c} aq/c,bq/c \\ q^2/c \end{array} ; q, z \right) \\
\qquad{} + (abz/c,cq/abz,bq/a;q) \theta(q/c;q) {}_2\phi_1 \left( \begin{array}{c} b,bq/c \\ bq/a\end{array} ; q, cq/abz \right) =0.
 \end{gather*}
The af\/f\/ine symmetry group is now given as the extension of the symmetry
group by also allowing elements which permute the four ${}_2\phi_1$'s
amongst themselves. Indeed the index $[W(A_3\times A_1) \colon W(A_2 \times
 A_1)]=4$. It is also the symmetry group of the $q$-dif\/ference equations
we discuss next.

The $q$-contiguous equations relate three terms of the form $B_{\alpha}^1(u
\cdot q^{\beta})$ where $\beta$ is in the projection $\Lambda_\alpha$ of
$\Lambda(E_7)$ on the space orthogonal to the face containing $\alpha$. In
particular the lattice $\Lambda_{\alpha}$ is generated by ($\pi$ denotes
the orthogonal projection on $\Lambda_\alpha$)
\begin{gather*}
 \pi(0,0,0,1,0,-1,0,0) = (1/4,0,0,3/4,-1/4,-1/2,-1/2,1/4),
 \\ \pi(0,0,0,0,1,-1,0,0) = (1/4,0,0,-1/4,3/4,-1/2,-1/2,1/4),
 \\ \pi(0,0,0,0,0,1,0,-1) = (1/4,0,0,-1/4,-1/4,1/2,1/2,-3/4),
 \\ \pi(0,0,1,0,0,-1,0,0) = (0,1/2,1/2,0,0,-1/2,-1/2,0).
\end{gather*}
If we simplify ${}_2\phi_1$ by setting $a=u_0u_3$, $b=u_0u_4$, $c=qu_0/u_7$
and $z=u_1u_2$, these four vectors correspond to multiplying respectively
$a$, $b$, $c$, or $z$ by $q$. In particular we have a relation for any
three sets of parameters where $a$, $b$, $c$ and $z$'s dif\/fer by an integer
power of $q$. For example using shifts $\pi(\rho - e_2-e_4-e_5-e_6)$,
$\pi(e_1-e_5)$ and $\pi(e_1-e_2)$ (corresponding to $a\mapsto aq$,
$z\mapsto qz$ and doing nothing), we get
\[
-(1-a) \rphisx{2}{1}{aq,b \\c }{z} - a\rphisx{2}{1}{a,b \\c}{qz} + \rphisx{2}{1}{a,b \\c}{z}=0.
\]

\section[Evaluations: the $m=0$ case]{Evaluations: the $\boldsymbol{m=0}$ case}\label{seceval}

In the $m=0$ case the general polytope picture is somewhat unsatisfying as
a description of the possible limits of the elliptic hypergeometric beta
integral evaluation. Indeed there are two issues. First of all the
polytope $P^{(0)}$ as described in Section~\ref{secpolytope} is not the
entire polytope for which proper limits exist, indeed we can give a larger
polytope for which this is true. Secondly, if we are interested in knowing
what the dif\/ferent evaluations on the basic hypergeometric level are, it
seems more natural to look at $P_{\rm I}^{(0)}$, $P_{\rm II}^{(0)}$ and
$P_{\rm III}^{(0)}$, and consider the faces of these polytopes, instead of
looking at $P^{(0)}$.

Let us f\/irst consider this second issue. In Section \ref{secpolytope} we
have actually shown that to each face of $P_{\rm I}$, $P_{\rm II}$ and $P_{\rm III}$
there is associated a function, which depends only on the space orthogonal
to that face. Moreover the iterated limit property holds in these
polytopes. If we therefore want to know what the dif\/ferent limit
evaluations are, we only have to write down the faces of these three
polytopes and the associated functions with their evaluations. As all
faces of these polytopes are simplicial, except for the interior of
$P_{\rm II}$, this is a simple combinatorial argument. All faces are listed in
the appendix in Fig.~\ref{degscheme2}. For those faces of $P$ which are
split in dif\/ferent faces of $P_{\rm I}$, $P_{\rm II}$ and $P_{\rm III}$ (i.e.\ the
interior and the facets given by $\alpha_r+\alpha_s+\alpha_t+\alpha_u=0$)
the value of the evaluation might be the same on all these dif\/ferent faces
of the smaller polytopes, but as the functions are a priori dif\/ferent, we
do obtain a dif\/ferent evaluation formula for each of these faces.

Now we look at the larger polytope.

\begin{Definition}
Def\/ine the extended polytope $P_{\rm ext}$ to be the polytope given as the
convex hull of the vectors $e_j$ ($0\leq j\leq 5$) and $f_j=\rho^{(0)} -
2e_j$ ($0\leq j\leq 5$).
\end{Definition}

The bounding inequalities are

\begin{Proposition}
The bounding inequalities of $P_{\rm ext}$ inside the subspace $2\rho^{(0)}\cdot \alpha=1$ are given~by
\begin{gather*}
 \alpha_r+\alpha_s \leq 1, \qquad (0\leq r<s\leq 5).
\end{gather*}
\end{Proposition}

\begin{proof}
This follows from a calculation as in Proposition~\ref{propboundingineq}
(though now we have $\binom{7}{2}$ options as we need to take two vectors
from seven).
\end{proof}

The following proposition follows immediately from the evaluation formula~\eqref{eqeval}.

\begin{Proposition}
For $\alpha \in P_{\rm ext}$ the limit \eqref{eqlim} exists and
$B_{\alpha}^0(u)$ is the same for each $\alpha$ in a~face of $P_{\rm ext}$ and
depends only on $u$ orthogonal to the face containing $\alpha$. Moreover
the iterated limit property holds.
\end{Proposition}

The question this immediately raises is to what extent one can give series
or integrals corresponding to points in $P_{\rm ext} \backslash P^{(0)}$. So
far we have not been able to give a good description of these limits, let
alone a classif\/ication. However we expect this is where we have to look
for evaluations of bilateral series. In the next section we consider more
generally what we expect.

\section{Going beyond the polytope}\label{secrq}
As indicated in the previous section there exist proper limits outside the
polytopes as described in Section \ref{secpolytope}. While we only know of
the existence of proper limits as in \eqref{eqlim} outside of $P^{(m)}$ in
the case $m=0$, if we let $p$ tend to zero along a geometric progression
and rescale the functions we can obtain limits in many more cases. In
general these limits will depend on what geometric progression we use for
$p$ (unlike in the $m=0$ case). We do not know for which points in
$\mathbb{R}^{2m+6}$ we can take limits in this way but it does seem to
provide a very rich extra set of functions. In particular, this seems to be
where bilateral series reside. For instance, Chen and Fu \cite{CF}
prove a~${}_2\psi_2$ transformation as a limit of (in our notation) a
transformation of $B^{(1)}_{(w^{(1)}_{01}+w^{(1)}_{02}+v^{(1)}_{67})/3}$,
taken in a direction pointing outside the Hesse polytope.

As an example we consider $m=1$ and a limit along $\alpha=v_1 + \epsilon
v_2$ for small $\epsilon>0$ and $v_1$ a vertex of $P^{(1)}$ and $v_2$ a
root of $E_7$ with $v_1 \cdot v_2=0$. Note that all of these limits are
related to each other by the Weyl group of $E_7$ action, so we expect to
obtain transformation formulas relating the functions associated to these
vectors. Up to permutations we have the following 6 options:
\begin{enumerate}\itemsep=0pt
\item $(0,0,0,0,0,0,1-\epsilon, 1+\epsilon)$;
\item $(-\epsilon/2, -\epsilon/2, -\epsilon/2, \epsilon/2, \epsilon/2,\epsilon/2,1-\epsilon/2,1+\epsilon/2)$;
\item $(-\epsilon,0,0,0,0,\epsilon, 1,1) $;
\item $(-1/2,-1/2,1/2-\epsilon,1/2,1/2,1/2,1/2,1/2+\epsilon)$;
\item $(-1/2-\epsilon/2, -1/2+\epsilon/2, 1/2-\epsilon/2,1/2-\epsilon/2,1/2-\epsilon/2,1/2-\epsilon/2,1/2-\epsilon/2)$;
\item $(-1/2-\epsilon, -1/2+\epsilon, 1/2,1/2,1/2,1/2,1/2,1/2)$.
\end{enumerate}

For some of these we can obtain limits as integrals. In the rest of this
section we suppose $\epsilon =1/N$ for some large integer $N$, and $p=x^N
q^{kN}$ (where $k$ is allowed to vary).
\begin{Proposition}
For $\alpha = (0,0,0,0,0,0,1-\epsilon,1+\epsilon)$ we have
\[
\lim_{k\to \infty} E_1(p^{\alpha} \cdot u) (xu_7)^{2k} q^{2 \binom{k}{2}}
= \prod_{0\leq r<s\leq 5} (u_ru_s;q) \frac{(q;q)}{2}
\int \frac{\theta(x u_7 z^{\pm 1};q) (z^{\pm 2};q) }{\prod\limits_{r=0}^5 (u_r z^{\pm 1};q)} \frac{dz}{2\pi iz}.
\]
\end{Proposition}
\begin{proof}
We calculate
\begin{gather*}
\lim_{k\to \infty} E_1(p^{\alpha} \cdot u) (xu_7)^{2k} q^{2 \binom{k}{2}} \\
\qquad{}
= \lim_{k\to \infty} (xu_7)^{2k} q^{2 \binom{k}{2}}
\prod_{0\leq r<s\leq 7} (p^{\alpha_r+\alpha_s} u_ru_s;p,q) \frac{(p;p)(q;q)}{2} \\
\qquad \qquad {}\times
\int \frac{\prod\limits_{r=0}^5 \Gamma(u_r z^{\pm 1}) \Gamma(p^{1-\epsilon} u_6 z^{\pm 1},p^{1+\epsilon} u_7 z^{\pm 1}) }{\Gamma(z^{\pm 2})} \frac{dz}{2\pi iz} \\
\qquad{}
=
\prod_{0\leq r<s\leq 5} (u_ru_s;q) \frac{(q;q)}{2}\lim_{k\to \infty} (xu_7)^{2k} q^{2 \binom{k}{2}} \\
 \qquad\qquad {}\times
 \int \frac{\theta(p^{\epsilon} u_7 z^{\pm 1};q) \prod\limits_{r=0}^5 \Gamma(u_r z^{\pm 1}) \Gamma(p^{1-\epsilon} u_6 z^{\pm 1},p^{\epsilon} u_7 z^{\pm 1}) }{\Gamma(z^{\pm 2})} \frac{dz}{2\pi iz} \\
\qquad{}
= \lim_{k\to \infty}
\prod_{0\leq r<s\leq 5} (u_ru_s;q) \frac{(q;q)}{2} \\
\qquad\qquad {}\times
\int \frac{\theta(x u_7 z^{\pm 1};q) \prod\limits_{r=0}^5 \Gamma(u_r z^{\pm 1}) \Gamma(p^{1-\epsilon} u_6 z^{\pm 1},p^{\epsilon} u_7 z^{\pm 1}) }{\Gamma(z^{\pm 2})} \frac{dz}{2\pi iz} \\
\qquad{}
= \prod_{0\leq r<s\leq 5} (u_ru_s;q) \frac{(q;q)}{2}
\int \frac{\theta(x u_7 z^{\pm 1};q) (z^{\pm 2};q) }{\prod\limits_{r=0}^5 (u_r z^{\pm 1};q)} \frac{dz}{2\pi iz}.
\end{gather*}
Here we used that all the poles are on the right side of the contour as
$p\to 0$, so we can interchange limit and integral as before. Moreover we
used that
\[
\theta(p^{\epsilon} y;q) = \theta(q^k x y;q) = \theta(xy;q) \left( -\frac{1}{xy} \right)^k q^{-\binom{k}{2}}
\]
for $y=u_7z$ and $y=u_7/z$.
\end{proof}

The next two limits are analogous.

\begin{Proposition}
For $\alpha =
(-1/2,-1/2,1/2-\epsilon,1/2,1/2,1/2,1/2,1/2+\epsilon)$ we have
\begin{gather*}
\lim_{k\to \infty} E_1(p^{\alpha} \cdot u) \left( \frac{qu_7 x^2}{u_0u_1u_2}\right)^k q^{2\binom{k}{2}}
=
(u_0u_1,q;q) \prod_{r=0}^1 \prod_{s=3}^6 (u_ru_s;q) \\
\qquad\qquad{} \times
\int \frac{(1-z^2) \theta(u_0u_1u_2/zx,xu_7/z;q) (qz/u_3,qz/u_4,qz/u_5,qz/u_6;q)}{(u_0z^{\pm 1}, u_1z^{\pm 1}, u_3z,u_4z,u_5z,u_6z;q)} \frac{dz}{2\pi i z}
\end{gather*}
and for $\alpha =
(-\epsilon/2,-\epsilon/2,-\epsilon/2,\epsilon/2,\epsilon/2,\epsilon/2,1-\epsilon/2,1+\epsilon/2)$
we have
\begin{gather*}
\lim_{k\to \infty} E_1(p^{\alpha} \cdot u) \left( \frac{qu_7
 x^2}{u_0u_1u_2}\right)^k q^{2\binom{k}{2}} =(q;q) \prod_{r=0}^2
\prod_{s=3}^5 (u_ru_s;q) \\
\qquad\qquad{} \times \int \frac{\theta(u_0u_1u_2/zx,u_7x/z;q)
 (q/u_7z,qz/u_6;q)}{(u_0/z,u_1/z,u_2/z,u_3z,u_4z,u_5z;q)} \frac{dz}{2\pi i
 z}.
\end{gather*}
\end{Proposition}

\begin{proof}
The f\/irst integral is a direct limit in the symmetry broken integral
\eqref{eqsymbreak3spec}, while the second limit comes from the symmetric
integral with $z\to p^{\epsilon/2} z$.
\end{proof}

For the other three limits we are unable to describe $B_{\alpha}$ using an
integral, but we do have series representations of these limits. As in the
case of Proposition~\ref{propsum} proofs of these limits involve tedious
calculations, so we just give a short sketch. As we have found is quite
common for series representations outside the polytope, we obtain bilateral
series. For example

\begin{Proposition}
For $\alpha = w_{01} + \epsilon(e_1-e_0)$ that if $|u_0u_1|<1$ we have
\begin{gather*}
\lim_{k\to \infty} E_1(p^{\alpha} \cdot u) x^{2k} q^{2\binom{k}{2}} \left(\frac{q}{u_0^2}\right)^k
=
\frac{(u_0u_1;q) }{(q;q) } \theta(u_0^2/x^2;q)
\prod_{r=2}^7 (qx/u_ru_0,qu_0/u_r x;q) \\ \qquad\qquad{}\times
{}_8\psi_8 \left( \begin{array}{c} \pm qu_0/x, u_2u_0/x, \ldots, u_7u_0/x \\
\pm u_0/x, qu_0/u_2x, \ldots, qu_0/u_7x \end{array} ; q, u_0u_1\right).
\end{gather*}
\end{Proposition}

\begin{proof}
To obtain this limit we look at the symmetry broken integral
\eqref{eqsymbreak2spec} with two $s_i$ speciali\-zed, and change the
integration variable $z\to p^{\epsilon-1/2}z$. Subsequently we pick up the
poles at $z=u_0p^{-2\epsilon}q^n$ (for $n=0$ to $n=2k$) and observe that
the remaining integral vanishes in the limit. Moreover the absolute value
of the summand in the sum of the residues is maximized near $n=k$ and we
can show that we can interchange limit and sum in $\sum_{n=-k}^k
\Res(z=u_0p^{-2\epsilon}q^{n+k})$, giving a bilateral sum.
\end{proof}

%\marginpar{[{\bf DONE}: Clarify this para.]}
Note that for $|u_0u_1|\geq 1$ we do not have an explicit expression for
the limit. The limit does have an analytic extension to the region
$|u_0u_1|\geq 1$, but we can only prove
this by using the Weyl group symmetry to relate the limit to the previous
limits obtained.

In the case $|u_0u_1|>1$ we could again use the general method of obtaining
a limit: discovering where the integrand or residues are maximized, rescaling
properly and interchanging limit and sum/integral. In this case it would lead to
\[%\begin{multline*}
\lim_{k\to \infty} E_1(p^{\alpha} \cdot u) x^{2k} q^{2\binom{k}{2}} \left(\frac{q}{u_0^3u_1}\right)^k
=
\frac{\prod\limits_{r=2}^7 \theta(u_rx/u_0;q)}{\theta(x^2/u_0^2;q)} \rphisx{1}{0}{u_0u_1 \\ -}{\frac{1}{u_0u_1}}
%\\ + \frac{(u_0u_1;q) \prod_{r=2}^7 \theta(u_0u_s/x;q)}{(q, u_0^2/x^2,u_1x^2/u_0;q)} \rphisx{2}{1}{u_1x^2/u_0,q \\ qx^2/u_0^2}{\frac{1}{u_0u_1}}
\]
%\end{multline*}
by picking up the residues at $z=u_0p^{-\epsilon}q^n$ in the same integral
as above but with $z\to p^{-1/2}z$ instead of $z\to
p^{\epsilon-1/2}z$. Note that the right hand side vanishes by the
evaluation formula for a $\rphis{1}{0}$. Indeed we can also see that this
limit vanishes by applying the Weyl group symmetry before taking the limit
and observing a factor $\lim_{k\to \infty} (1/u_0u_1)^k=0$ remains after
using one of the integral limits above. In particular this shows one has to
be careful taking limits in order to obtain something interesting.

\begin{Proposition}
For $\alpha = (-\epsilon,0,0,0,0,\epsilon,1,1)$ and $|u_0u_5|<1$ we obtain
the limit
\begin{gather*}
\lim_{k\to \infty} E^1(p^{\alpha} \cdot u) \left( \frac{u_5u_6u_7}{u_0}\right)^k x^{2k} q^{2\binom{k}{2}}
 =
\theta(u_0u_4/x,u_0u_1u_2 u_3/x ;q) \\
 \qquad\qquad{} \times \frac{(u_0u_5,u_1u_4,u_2u_4,qu_3/u_1, qu_3/u_2, qu_3/u_6, q/u_3u_6,qu_3/u_7,q/u_3u_7;q)}
{(q/u_1u_2, qu_3^2 , u_4 /u_3;q)} \\
 \qquad \qquad {} \times
{}_8W_7( u_3^2; u_1u_3,u_2u_3,u_4u_3,u_6u_3,u_7u_3; u_0u_5 )
+ (u_3\leftrightarrow u_4).
\end{gather*}
Moreover for $\alpha = w_{01} + \epsilon (\rho - e_0-e_2-e_3-e_4)$ $($without convergence conditions$)$
\begin{gather*}
\lim_{k\to \infty} E^1(p^{\alpha} \cdot u) x^{2k} q^{2\binom{k}{2}} \left( \frac{u_5u_6u_7 }{u_0}\right)^k \\
\qquad{} =
 \frac{ (u_0u_1;q) \prod\limits_{r=2}^4 ( qx/u_0u_r,u_1u_r;q) \prod\limits_{r=5}^7 (qu_0 /xu_r ,u_0u_r;q)}
{(q,u_0^2 /x, x u_1/u_0;q)} \\ \qquad \qquad {} \times
{}_4\psi_4\left( \begin{array}{c} u_0^2/x, u_0u_2/x,u_0u_3/x,u_0u_4/x \\ qu_0/xu_1, qu_0/xu_5,qu_0/xu_6,qu_0/xu_7 \end{array} ;q,q \right)
\\ \qquad\qquad {}+
\frac{\prod\limits_{r=2}^4 \theta( u_0u_r/x;q) }{\theta(u_0/xu_1;q)}
\prod\limits_{r=5}^7 (qu_1/u_r,u_0u_r;q)
{}_4\phi_3 \left( \begin{array}{c} u_0u_1, u_2u_1,u_3u_1,u_4u_1 \\ qu_1/u_5,qu_1/u_6,qu_1/u_7 \end{array} ;q,q\right).
\end{gather*}
\end{Proposition}
\begin{proof}
For the f\/irst limit, use the symmetry broken integral
\eqref{eqsymbreak3spec} with three specializations, shift $z\to
p^{\epsilon}z $ and pick up the residues at $z=p^{\epsilon} q^n/u_3$ and
$z=p^{\epsilon}q^n/u_4$. We arbitrarily broke the symmetry between $u_1$
and $u_2$, with $u_3$ and $u_4$ here to be able to write this as the sum of
only two series.

The second limit is obtained by picking up the residues at
$z=p^{-3\epsilon/2} u_0q^n$ (for the ${}_4\psi_4$) and $z=p^{-\epsilon/2}
u_1 q^n$ (for the ${}_4\phi_3$) in the symmetry broken integral
\eqref{eqsymbreak2spec} with two specializations, with $z\to
p^{-1/2+\epsilon}$.
\end{proof}

Now we have obtained these limits we can obtain relations for these
functions as before. In particular we obtain the symmetries of these
functions (the symmetry group, the stabilizer of~$\alpha$ in~$W(E_7)$, is
isomorphic to $W(A_5)$),
and we obtain transformation formulas relating all these six functions in
terms of each other. This includes the formula \cite[(III.38)]{GandR}
expressing an ${}_8\psi_8$ in two very-well-poised ${}_8\phi_7$'s.

As mentioned before the big dif\/ference between the limits inside the
polytope and those outside is that inside we do not need to specialize $p$
to a geometric progression. In particular the functions just obtained do
depend non-trivially on the parameter $x$; $x$ is not just a cosmetic
factor necessary to calculate the limit as $w$ was in Proposition
\ref{propsumintegral}. The simplest way to see this is to specialize to an
evaluation formula. In the elliptic beta integral $E^1$ we reduce to~$E^0$
if the product of two parameters equals $pq$. Thus we must f\/ind a pair $r$,
$s$ such that $\alpha_r+\alpha_s=1$ and set $u_ru_s=q$. This is not
possible for all of the limits, but for example in the case $\alpha =
(-\epsilon/2,-\epsilon/2,-\epsilon/2,\epsilon/2,\epsilon/2,\epsilon/2,
1-\epsilon/2,1+\epsilon/2)$ we can set $u_5u_6=q$. The limit for the
evaluation formula works in precisely the same way, and we obtain
\begin{gather*}
(q/u_0u_7,q/u_1u_7,q/u_2u_7;q) \theta(q/xu_3u_7,q/xu_4u_7;q)
\\ \qquad{} =(q;q) \prod_{r=0}^2 \prod_{s=3}^5 (u_ru_s;q)
\int \frac{\theta(u_0u_1u_2/zx,u_7x/z;q) (q/u_7z;q)}{(u_0/z,u_1/z,u_2/z,u_3z,u_4z;q)} \frac{dz}{2\pi i z}.
\end{gather*}
The left hand side is clearly non-trivially dependent on $x$, while the
right hand side is the (rescaled) limit for $\alpha$ specialized in
$u_5u_6=q$.

\newpage

\begin{landscape}
\appendix

\section[The $m=1$ limits]{The $\boldsymbol{m=1}$ limits}\label{secapp1}

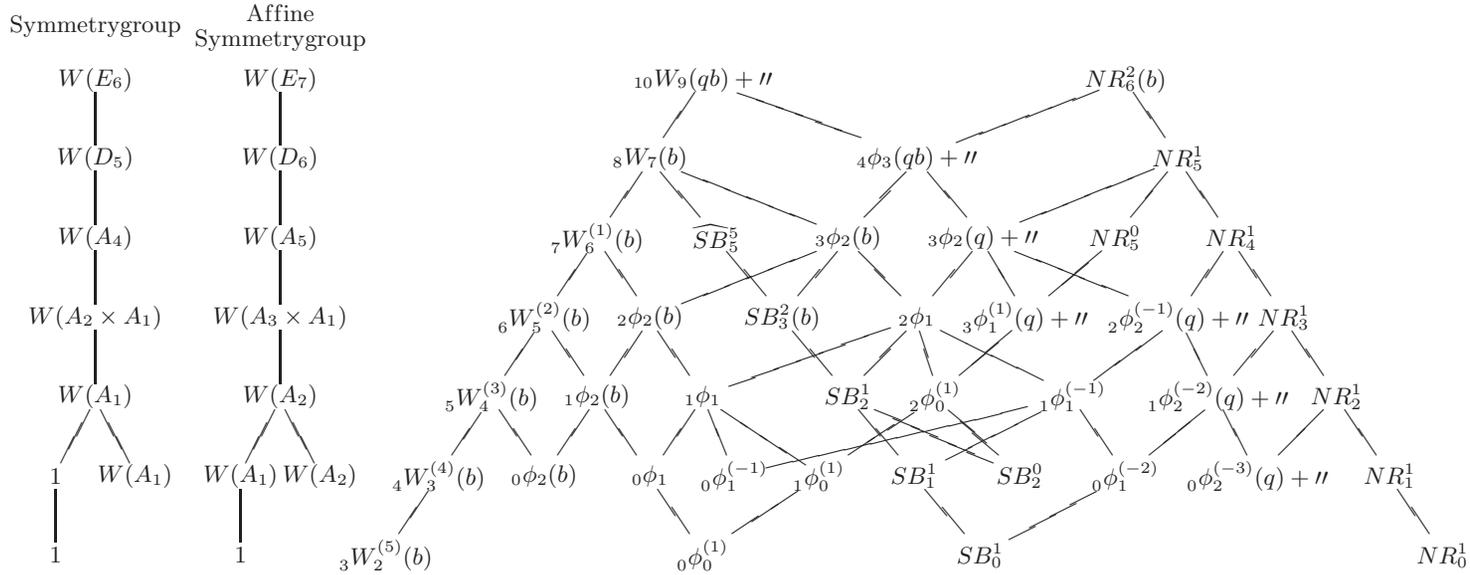
\begin{figure}[h]
\begin{center}
\begin{picture}(500,220)(-110,30)

\footnotesize

%%%%%%%%%%%%%%%%%%%%%%%%%%%%%%%%%%%%%%%%%%%%%%%TAKE 2 %%%%%%%%%%%%%%%%%%%%%%%%%%%%%%%%%%%%%%%%%%%

\mput{0,30}{$\rWsn{3}{2}{5}(b)$} % w01 w02 w03 w04 w05 w06 w07
\mput{120,30}{$\rphisn{0}{0}{1}$} % w01 w02 w03 w04 v56 v57 v67
\mput{225,30}{$SB_0^1$} % w01 w02 w12 v37 v47 v57 v67
\mput{400,30}{$NR_0^1$} % v07 v17 v27 v37 v47 v57 v67

\connect{400}{380}{30}{8}{5}
\connect{225}{280}{30}{6}{6}
\connect{225}{200}{30}{8}{5}
\connect{120}{164}{30}{8}{6}
\connect{120}{100}{30}{8}{5}
\connect{0}{20}{30}{8}{6}

\mput{20,60}{$\rWsn{4}{3}{4}(b)$} % w01 w02 w03 w04 w05 w06
\mput{60,60}{${}_0\phi_2(b)$} % w01 w02 w03 w04 w05 v67
\mput{100,60}{${}_0\phi_1$} % w01 w02 w03 w04 v57 v67
\mput{132,60}{$\rphisn{0}{1}{-1}$} % w01 w02 w03 v47 v57 v67
\mput{164,60}{$\rphisn{1}{0}{1}$} % w01 w02 w03 v56 v57 v67
\mput{200,60}{$SB_1^1$} % w01 w02 w12 v47 v57 v67
\mput{240,60}{$SB_2^0$} % w01 w02 w12 v56 v57 v67
\mput{280,60}{$\rphisn{0}{1}{-2}$} % w01 w02 v37 v47 v57 v67
\mput{330,60}{$\rphisn{0}{2}{-3}(q)+\ditto$}% + \rphisn{0}{2}{-3}(q)$} % w01 v27 v37 v47 v57 v67
\mput{380,60}{$NR_1^1$} % v17 v27 v37 v47 v57 v67

\connect{380}{360}{60}{8}{5}
\connect{330}{360}{60}{8}{5}
\connect{330}{315}{60}{8}{5}
\connect{280}{315}{60}{8}{5}
\connect{280}{260}{60}{8}{5}
\connect{240}{208}{60}{8}{5}
\connect{240}{175}{60}{5}{5}
\connect{200}{260}{60}{6}{6}
\connect{200}{175}{60}{8}{5}
\connect{164}{208}{60}{8}{6}
\connect{164}{120}{60}{4}{5}
\connect{132}{260}{60}{3}{4}
\connect{132}{120}{60}{8}{5}
\connect{100}{120}{60}{8}{5}
\connect{100}{80}{60}{8}{5}
\connect{60}{80}{60}{8}{5}
\connect{60}{40}{60}{8}{5}
\connect{20}{40}{60}{8}{6}

\mput{40,90}{$\rWsn{5}{4}{3}(b)$} % w01 w02 w03 w04 w05
\mput{80,90}{${}_1\phi_2(b)$} % w01 w02 w03 w04 v67
\mput{120,90}{${}_1\phi_1$} % w01 w02 w03 v57 v67
\mput{175,90}{$SB_2^1$} % w01 w02 w12 v57 v67
\mput{208,90}{$\rphisn{2}{0}{1}$} % w01 w02 v56 v57 v67
\mput{260,90}{$\rphisn{1}{1}{-1}$} % w01 w02 v47 v57 v67
\mput{315,90}{$\rphisn{1}{2}{-2}(q) + \ditto$}%\rphisn{1}{2}{-2}(q)$} % w01 v37 v47 v57 v67
\mput{360,90}{$NR_2^1$} % v27 v37 v47 v57 v67

\connect{360}{340}{90}{8}{5}
\connect{315}{340}{90}{8}{5}
\connect{315}{300}{90}{8}{6}
\connect{260}{300}{90}{6}{5}
\connect{260}{201}{90}{6}{5}
\connect{208}{242}{90}{8}{5}
\connect{208}{201}{90}{8}{5}
\connect{175}{201}{90}{8}{5}
\connect{175}{150}{90}{8}{5}
\connect{120}{201}{90}{4}{4}
\connect{120}{100}{90}{8}{5}
\connect{80}{100}{90}{8}{5}
\connect{80}{60}{90}{8}{5}
\connect{40}{60}{90}{8}{6}

\mput{60,120}{$\rWsn{6}{5}{2}(b)$} % w01 w02 w03 w04
\mput{100,120}{${}_2\phi_2(b)$} % w01 w02 w03 v67
\mput{150,120}{$SB_3^2(b)$} % w01 w02 w12							v67
\mput{201,120}{${}_2\phi_1$}													% w01 w02							v57 v67
\mput{242,120}{$\rphisn{3}{1}{1}(q) + \ditto$}% + \rphisn{3}{1}{1}(q)$}			% w01 						v56 v57 v67
\mput{300,120}{$\rphisn{2}{2}{-1}(q) + \ditto$}% + \rphisn{2}{2}{-1}(q)$}		% w01 						v47 v57 v67
\mput{340,120}{$NR_3^1$}																						% 						 v37 v47 v57 v67

\connect{340}{320}{120}{8}{5}
\connect{300}{320}{120}{8}{5}
\connect{300}{226}{120}{8}{5}
\connect{242}{276}{120}{8}{5}
\connect{242}{226}{120}{8}{5}
\connect{201}{226}{120}{8}{5}
\connect{201}{175}{120}{8}{5}
\connect{150}{175}{120}{8}{5}
\connect{150}{125}{120}{8}{5}
\connect{100}{175}{120}{5}{5}
\connect{100}{80}{120}{8}{5}
\connect{60}{80}{120}{8}{6}

\mput{80,150}{$\rWsn{7}{6}{1}(b)$}																			% w01 w02 w03
\mput{125,150}{$\widehat{SB}{}^5_5$}																		% w01 w02 w12
\mput{175,150}{${}_3\phi_2(b)$}												% w01 w02 								 v67
\mput{226,150}{${}_3\phi_2(q) +\ditto$}% + {}_3\phi_2(q)$}								% w01 								 v57 v67
\mput{276,150}{$NR_5^0$}																						% 								 v56 v57 v67
\mput{320,150}{$NR_4^1$}																						% 								 v47 v57 v67

\connect{320}{300}{150}{8}{5}
\connect{276}{300}{150}{8}{5}
\connect{226}{300}{150}{6}{4}
\connect{226}{201}{150}{8}{5}
\connect{175}{201}{150}{8}{5}
\connect{175}{100}{150}{5}{5}
\connect{125}{100}{150}{8}{5}
\connect{80}{100}{150}{8}{5}

\mput{100,180}{${}_8W_7(b)$}														% w01 w02
\mput{201,180}{${}_4\phi_3(qb) + \ditto % + {}_4\phi_3(qb)
$}									% w01 										v67
\mput{300,180}{$NR_5^1$} 																					% 										 v57 v67

\connect{300}{280}{180}{8}{5}
\connect{201}{280}{180}{6}{4}
\connect{201}{120}{180}{8}{5}
\connect{100}{120}{180}{8}{5}

\mput{120,210}{${}_{10}W_9(qb) + \ditto % + {}_{10}W_9(qb)
$}					% w01
\mput{280,210}{$NR_6^2(b)$}														% 												 v67

\mput{-125,30}{$1$}
\mput{-95,60}{$W(A_1)$} \mput{-125,60}{$1$}
\mput{-110,90}{$W(A_1)$}
\mput{-110,120}{$W(A_2 \times A_1)$}
\mput{-110,150}{$W(A_4)$}
\mput{-110,180}{$W(D_5)$}
\mput{-110,210}{$W(E_6)$}

\connect{-125}{-125}{30}{5}{5}
\connect{-95}{-110}{60}{5}{5}
\connect{-125}{-110}{60}{5}{5}
\connect{-110}{-110}{90}{5}{5}
\connect{-110}{-110}{120}{5}{5}
\connect{-110}{-110}{150}{5}{5}
\connect{-110}{-110}{180}{5}{5}

\mput{-110,230}{Symmetrygroup}

\mput{-40,235}{Af\/f\/ine}
\mput{-40,225}{Symmetrygroup}

\mput{-40,210}{$W(E_7)$}
\mput{-40,180}{$W(D_6)$}
\mput{-40,150}{$W(A_5)$}
\mput{-40,120}{$W(A_3 \times A_1)$}
\mput{-40,90}{$W(A_2)$}
\mput{-25,60}{$W(A_2)$} \mput{-55,60}{$W(A_1)$}
\mput{-55,30}{$1$}

\connect{-55}{-55}{30}{5}{5}
\connect{-25}{-40}{60}{5}{5}
\connect{-55}{-40}{60}{5}{5}
\connect{-40}{-40}{90}{5}{5}
\connect{-40}{-40}{120}{5}{5}
\connect{-40}{-40}{150}{5}{5}
\connect{-40}{-40}{180}{5}{5}

\normalsize

\end{picture}
\end{center}
\caption{The simplicial faces of $P^{(1)}$.}\label{degscheme1}

\end{figure}
Fig.~\ref{degscheme1} shows the functions associated to the ($S_8$-orbits
of) simplicial faces of $P^{(1)}$. We connect two faces by an edge if one
is a facet of the other; the degenerations of a given function are those
connected to it by a downward path in the graph.

To simplify the picture, we omitted the non-simplicial faces, i.e., the
interior together with the facets cut out by equations of the form
$\alpha_0+\alpha_1+\alpha_2+\alpha_3=0$ or $\alpha_7=1+\alpha_0$. Note
that these non-simplicial faces all correspond to evaluations.

In the scheme we used the following abbreviations for integrals
\begin{gather*}
NR_{a}^b = \int \frac{(z^{\pm 2};q) \prod\limits_{r=1}^b (w_r z^{\pm 1};q)}{\prod\limits_{r=1}^a (v_rz^{\pm 1};q)} \frac{dz}{2\pi i z}, \qquad
SB^b_a = \int \theta(u/z;q) \frac{\prod\limits_{r=1}^b (w_rz;q)}{\prod\limits_{r=1}^a (v_rz;q)}
(1-z^2)^{1_{\widehat{SB}}} \frac{dz}{2\pi i z}.
\end{gather*}
Here $1_{\widehat{SB}}$ denotes 1 if we consider $\widehat{SB}$ and 0 otherwise. We write $\rphis{r}{s} + \ditto$ or ${}_r W_{r-1}+\ditto$ to
denote the sum of two series with related coef\/f\/icients (as in Proposition~\ref{propsymbreak}).

After series we write $(q)$ if $z=q$ in the series, $(b)$ if the balancing condition $z \prod a_r = \prod b_r$ should hold, and
$(qb)$ if the series is balanced (i.e.\ both the above properties hold).
%Moreover the term $+''$ implies we have to add a similar second series.
Similarly $S_3^2(b)$ and $NR_6^2(b)$ indicate that a balancing condition holds amongst their parameters.

We want to stress that all functions are entire in their parameters. In
particular for the non-conf\/luent series without the condition $z=q$ (which
might fail to converge) we have an integral representation which gives an
analytic extension to all values of $z$.

\begin{figure}[h]
\begin{center}
\setlength{\unitlength}{1.25pt}
\begin{picture}(400,160)

\mput{210,0}{$SB_0^0$} % w01 w02 w12 v47 v57 v67
\mput{390,0}{$NR_0^0$} % v17 v27 v37 v47 v57 v67

\connect{390}{360}{0}{8}{5}
\connect{210}{240}{0}{8}{5}
\connect{210}{180}{0}{8}{5}

\mput{0,30}{$\rWsn{3}{2}{3}(b)$} % w01 w02 w03 w04 w05
\mput{60,30}{${}_0\phi_1(b)$} % w01 w02 w03 w04 v67
\mput{120,30}{${}_0\phi_0$} % w01 w02 w03 v57 v67
\mput{180,30}{$SB_1^0$} % w01 w02 w12 v57 v67
\mput{240,30}{$\rphisn{0}{0}{-1}$} % w01 w02 v57 v67
\mput{300,30}{$\rphisn{0}{1}{-2}(q) +\ditto$}% \rphisn{0}{1}{-2}(q)$} % w01 v37 v47 v57 v67
\mput{360,30}{$NR_1^0$} % v27 v37 v47 v57 v67

\connect{360}{330}{30}{8}{5}
\connect{300}{330}{30}{8}{5}
\connect{300}{270}{30}{8}{5}
\connect{240}{270}{30}{8}{5}
\connect{240}{210}{30}{8}{5}
\connect{180}{210}{30}{8}{5}
\connect{180}{150}{30}{8}{5}
\connect{120}{210}{30}{5}{5}
\connect{120}{90}{30}{8}{5}
\connect{60}{90}{30}{8}{5}
\connect{60}{30}{30}{8}{5}
\connect{0}{30}{30}{8}{5}

\mput{30,60}{$\rWsn{4}{3}{2}(b)$} % w01 w02 w03 w04
\mput{90,60}{${}_1\phi_1(b)$} % w01 w02 w03 v67
\mput{150,60}{$SB_2^1(b)$} % w01 w02 w12							v67
\mput{210,60}{${}_1\phi_0$}													% w01 w02							v57 v67
\mput{270,60}{$\rphisn{1}{1}{-1}(q) +\ditto$}% \rphisn{1}{1}{-1}(q)$}		% w01 						v47 v57 v67
\mput{330,60}{$NR_2^0$}																						% 						 v37 v47 v57 v67

\connect{330}{300}{60}{8}{5}
\connect{270}{300}{60}{8}{5}
\connect{270}{240}{60}{8}{5}
\connect{210}{240}{60}{8}{5}
\connect{210}{180}{60}{8}{5}
\connect{150}{180}{60}{8}{5}
\connect{150}{120}{60}{8}{5}
\connect{90}{180}{60}{5}{5}
\connect{90}{60}{60}{8}{5}
\connect{30}{60}{60}{8}{5}

\mput{60,90}{$\rWsn{5}{4}{1}(b)$}																			% w01 w02 w03
\mput{120,90}{$\widehat{SB}{}^3_3$}																		% w01 w02 w12
\mput{180,90}{${}_2\phi_1(b)$}												% w01 w02 								 v67
\mput{240,90}{${}_2\phi_1(q) + \ditto$}% {}_2\phi_1(q)$}								% w01 								 v57 v67
\mput{300,90}{$NR_3^0$}																						% 								 v47 v57 v67

\connect{300}{270}{90}{8}{5}
\connect{240}{270}{90}{8}{5}
\connect{240}{210}{90}{8}{5}
\connect{180}{210}{90}{8}{5}
\connect{180}{150}{90}{8}{5}
\connect{120}{150}{90}{8}{5}
\connect{60}{150}{90}{5}{5}

\mput{150,120}{${}_6W_5(b)$}														% w01 w02
\mput{210,120}{${}_3\phi_2(qb) +\ditto$}% {}_3\phi_2(qb)$}									% w01 										 v67
\mput{270,120}{$NR_4^0$} 																					% 										 v57 v67

\connect{270}{240}{120}{8}{5}
\connect{210}{240}{120}{8}{5}
\connect{210}{180}{120}{8}{5}
\connect{150}{180}{120}{8}{5}

\mput{180,150}{${}_{8}W_7(qb) +\ditto$}% + {}_{8}W_7(qb)$}					% w01
\mput{240,150}{$NR_5^1(b)$}														% 												 v67

\end{picture}
\end{center}
\caption{The simplicial faces of $P_{\rm I}^{(0)}$, $P_{\rm II}^{(0)}$ and $P_{\rm III}^{(0)}$.}\label{degscheme2}
\end{figure}
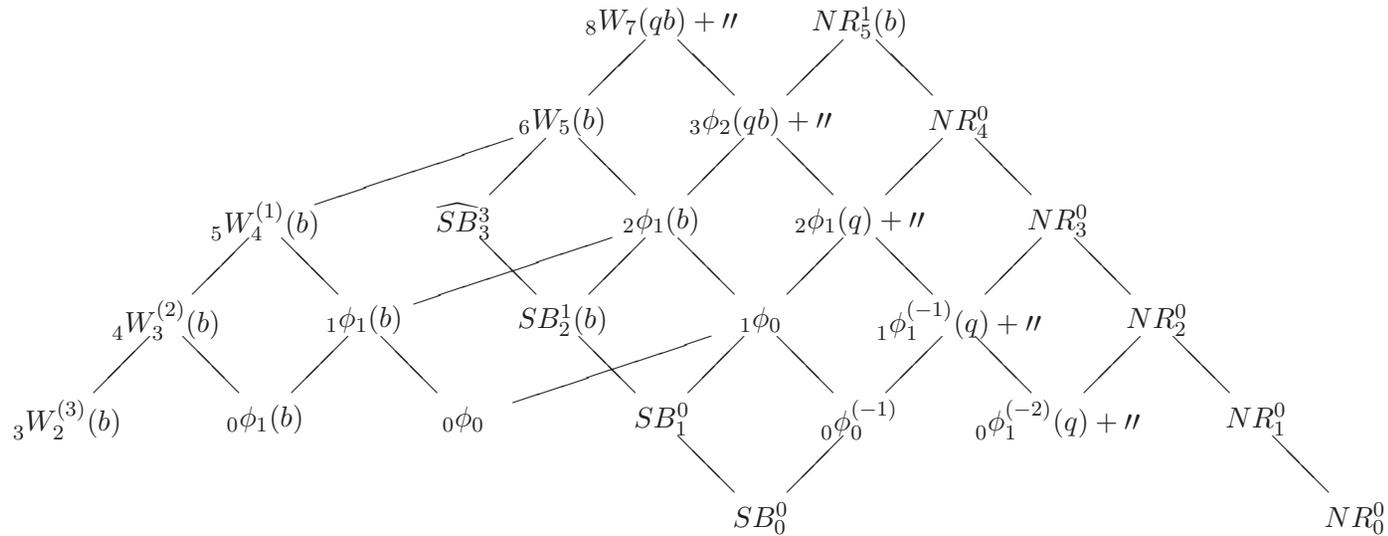

In Fig.~\ref{degscheme2} we have depicted the functions corresponding to
the simplicial faces of $P_{\rm I}^{(0)}$, $P_{\rm II}^{(0)}$, $P_{\rm III}^{(0)}$,
i.e.\ those functions for which we obtain an evaluation formula. The only
non-simplicial face of these three polytopes is the interior of $P_{\rm II}$
(on which the function given by the relevant list of functions, Proposition
\ref{propsum}, equals 1 identically). The faces of $P_{\rm I}$ are those which
have $NR_{0}^0$ as a limit; i.e., the limits of the form $NR^{*}_*$.
Similarly, the faces of $P_{\rm III}$ are those which have $SB_0^0$ as a limit,
and any function not of the form $NR_0^0$ or $SB$ corresponds to a
simplicial face of $P_{\rm II}$.

\end{landscape}

\subsection*{Acknowledgements}

The second author was supported in part by NSF grant DMS-0833464.

\pdfbookmark[1]{References}{ref}
\LastPageEnding


\begin{thebibliography}{99}

\footnotesize\itemsep=0pt


\bibitem{vdBRS}
van de Bult F.J., Rains E.M., Stokman J.V.,
 Properties of generalized univariate hypergeometric functions,
 \textit{Comm. Math. Phys.} \textbf{275} (2007), 37--95,
 \href{http://arxiv.org/abs/math.CA/0607250}{math.CA/0607250}.

\bibitem{CF}
Chen W.Y.C., Fu A.M.,
Semi-f\/inite forms of bilateral basic hypergeometric series,
\textit{Proc. Amer. Math. Soc.} \textbf{134} (2006), 1719--1725,
\href{http://arxiv.org/abs/math.CA/0501242}{math.CA/0501242}.

\bibitem{CS}
Conway J.H., Sloane N.J.A.,
The cell structures of certain lattices, in Miscellanea Mathematica, Springer, Berlin, 1991, 71--107.

\bibitem{vDS}
van Diejen J.F., Spiridonov V.P.,
An elliptic Macdonald--Morris conjecture and multiple modular hyper\-geo\-metric sums,
\textit{Math. Res. Lett.} \textbf{7} (2000), 729--746.

\bibitem{vDS2}
van Diejen J.F., Spiridonov V.P.,
Elliptic Selberg integrals,
\textit{Internat. Math. Res. Notices} \textbf{2001} (2001), no.~20, 1083--1110.

\bibitem{GandR}
Gasper G., Rahman M.,
Basic hypergeometric series, 2nd ed., {\it Encyclopedia of Mathematics and Its Applications}, Vol.~96, Cambridge University Press, Cambridge, 2004.

\bibitem{GM}
Gupta D.P., Masson D.R.,
Contiguous relations, continued fractions and orthogonality,
\textit{Trans. Amer. Math. Soc.} \textbf{350} (1998), 769--808,
\href{http://arxiv.org/abs/math.CA/9511218}{math.CA/9511218}.

\bibitem{LvdJ1}
Lievens S., Van der Jeugd J.,
Invariance groups of three term transformations for basic hypergeometric series,
\textit{J. Comput. Appl. Math.} \textbf{197} (2006), 1--14.

\bibitem{LvdJ}
Lievens S., Van der Jeugd J.,
Symmetry groups of Bailey's transformations for ${}_{10}\phi_9$-series,
\textit{J. Comput. Appl. Math.} \textbf{206} (2007), 498--519.

\bibitem{e7trafo}
 Rains E.M.,
 Transformations of elliptic hypergeometric integrals,
 \textit{Ann. Math.}, to appear,
\href{http://arxiv.org/abs/math.QA/0309252}{math.QA/0309252}.

\bibitem{limits}
 Rains E.M.,
Limits of elliptic hypergeometric integrals,
\textit{Ramanujan J.}, to appear,
\href{http://arxiv.org/abs/math.CA/0607093}{math.CA/0607093}.

\bibitem{eli}
Rains E.M.,
Elliptic Littlewood identities,
\href{http://arxiv.org/abs/0806.0871}{arXiv:0806.0871}.

\bibitem{ellgamma}
 Ruijsenaars S.N.M.,
First order analytic dif\/ference equations and integrable quantum systems,
\textit{J. Math. Phys.} \textbf{38} (1997), 1069--1146.


\bibitem{evalform}
 Spiridonov V.P.,
On the elliptic beta function,
{\it Uspekhi Mat. Nauk} {\bf 56} (2001), no.~1 (337), 181--182 (English transl.: {\it Russian Math. Surveys} {\bf 56} (2001), no.~1, 185--186).

\bibitem{Spirthi}
Spiridonov V.P.,
Theta hypergeometric integrals,
\textit{Algebra i Analiz} \textbf{15} (2003), 161--215 (English transl.: \textit{St. Petersburg Math. J.} \textbf{15} (2004), 929--967),
\href{http://arxiv.org/abs/math.CA/0303205}{math.CA/0303205}.

\bibitem{Spircehf}
 Spiridonov V.P.,
 Classical elliptic hypergeometric functions and their applications, {\it Rokko Lect. in Math.}, Vol.~18, Kobe University, 2005, 253--287,
 \href{http://arxiv.org/abs/math.CA/0511579}{math.CA/0511579}.

\bibitem{Spirspeb}
Spiridonov V.P.,
Short proofs of the elliptic beta integrals,
\textit{Ramanujan J.} \textbf{13} (2007), 265--283,
\href{http://arxiv.org/abs/math.CA/0408369}{math.CA/0408369}.

\bibitem{Spiress}
Spiridonov V.P.,
Essays on the theory of elliptic hypergeometric functions,
\textit{Uspekhi Mat. Nauk} {\bf 63} (2008), no.~3, 3--72 (English transl.: \textit{Russian Math. Surveys} {\bf 63} (2008), no.~3, 405--472),
\href{http://arxiv.org/abs/0805.3135}{arXiv:0805.3135}.



\end{thebibliography}
\end{document}